\documentclass[a4paper,leqno]{amsart}

\usepackage[english]{babel} 									
\usepackage[T1]{fontenc}							
\usepackage[utf8]{inputenx}
				
\usepackage{mathtools}								
\usepackage{empheq}									
\usepackage{enumitem}								

\usepackage{amssymb}								
\usepackage[sc]{mathpazo}							
\usepackage{mathrsfs}								

\usepackage{verbatim}                               
\usepackage{iflang}									
\usepackage{xifthen}								

\usepackage[final]{pdfpages}						
\usepackage{todonotes}								
\usepackage{aliascnt}								
\usepackage{needspace}								
\usepackage{refcount}

\usepackage{subcaption}			
\usepackage{placeins} 			
\usepackage{grffile}			
\usepackage[all]{xy}			

\usepackage{transparent}
\usepackage{color}				

\usepackage{graphicx}			

\usepackage[babel]{csquotes}							

\usepackage{nameref}									
\usepackage[colorlinks	=	true,						
          	 	linkcolor	=	red,
            	urlcolor	=	red,
            	citecolor	=	red]%
            	{hyperref}
\usepackage{bookmark}									
\usepackage{cleveref}									
\theoremstyle{definition}
\newtheorem{thm}{Theorem}[section]

\newtheorem{defi}[thm]{Definition}
\newtheorem{lemm}[thm]{Lemma}

\newtheorem{rem}[thm]{Remark}

\newtheorem{prop}[thm]{Proposition}
\newtheorem{ques}[thm]{Question}

\theoremstyle{remark}

\DeclareMathOperator{\Mod}{mod}

\DeclareMathOperator{\diam}{diam}

\DeclareMathOperator{\loc}{loc}

\allowdisplaybreaks

\newcommand{\norm}[1]{ \left\Vert #1 \right\Vert }	
\newcommand{\abs}[1]{ \left| #1 \right| }           
\newcommand{\apmd}[2][]{							
	\ifthenelse{\equal{#1}{}}%
					{ \operatorname{N}_{#2}	}%
					{ \operatorname{N}_{#1,#2} 	}}

\begin{document}
\title[Metric spheres from gluing hemispheres]{Two-dimensional metric spheres from gluing hemispheres}

\author{Toni Ikonen}

\address{University of Jyvaskyla \\ Department of Mathematics and Statistics \\
P.O. Box 35 (MaD) \\
FI-40014 University of Jyvaskyla}
\email{toni.m.h.ikonen@jyu.fi}

\thanks{The author was supported by the Academy of Finland, project number 308659 and by the Vilho, Yrjö and Kalle Väisälä Foundation.}



\begin{abstract}
We study metric spheres $( Z, d_{Z} )$ obtained by gluing two hemispheres of $\mathbb{S}^{2}$ along an orientation-preserving homeomorphism $g \colon \mathbb{S}^{1} \rightarrow \mathbb{S}^{1}$, where $d_{Z}$ is the canonical distance that is locally isometric to $\mathbb{S}^{2}$ off the seam.

We show that if $( Z, d_{Z} )$ is quasiconformally equivalent to $\mathbb{S}^{2}$, in the geometric sense, then $g$ is a welding homeomorphism with conformally removable welding curves. We also show that $g$ is bi-Lipschitz if and only if $( Z, d_{Z} )$ has a $1$-quasiconformal parametrization whose Jacobian is comparable to the Jacobian of a quasiconformal mapping $h \colon \mathbb{S}^{2} \rightarrow \mathbb{S}^{2}$. Furthermore, we show that if $g^{-1}$ is absolutely continuous and $g$ admits a homeomorphic extension with exponentially integrable distortion, then $( Z, d_{Z} )$ is quasiconformally equivalent to $\mathbb{S}^{2}$.
\end{abstract}

\maketitle\thispagestyle{empty}

\section{Introduction}\label{sec:intro}
    In this paper, we work in the unit sphere $\mathbb{S}^{2} \subset \mathbb{R}^{3}$. We denote the equator $\mathbb{S}^{2} \cap \left( \mathbb{R}^{2} \times \left\{0\right\} \right)$ by $\mathbb{S}^{1}$ and endow $\mathbb{S}^{2}$ with the length distance $\sigma$ induced by the Euclidean distance of $\mathbb{R}^{3}$. The open southern and northern hemispheres are denoted by $Z_{1}$ and $Z_{2}$, respectively. Here $( 0, 0, 1 ) \in Z_{2}$.

    Consider an orientation-preserving homeomorphism $g \colon \mathbb{S}^{1} \rightarrow \mathbb{S}^{1}$, mapping the boundary of $\overline{Z}_{1}$ to the boundary of $\overline{Z}_{2}$. We identify each $z \in \mathbb{S}^{1}$ with its image $g(z) \in \mathbb{S}^{1}$. With this identification, we obtain a set $Z$ and inclusion maps $\iota_{1} \colon \overline{ Z }_{1} \rightarrow Z$ and $\iota_{2} \colon \overline{ Z }_{2} \rightarrow Z$. We call $S_{Z} = \iota_{1}( \mathbb{S}^{1} ) = \iota_{2}( \mathbb{S}^{1} )$ the \emph{seam} of $Z$.

    We construct a pseudodistance $d_{Z}$ on $Z$, see \Cref{sec:sew}, making the inclusion maps local isometries off the seam and $1$-Lipschitz everywhere. We consider the quotient map $Q \colon Z \rightarrow \widetilde{Z}$ identifying points $x, y \in Z$ whenever $d_{Z}( x, y ) = 0$, and endow $\widetilde{Z}$ with the associated quotient distance.
    
    We are interested in this construction for the following reason: whenever the metric space $\widetilde{Z}$ is quasiconformally equivalent to $\mathbb{S}^{2}$, there exist Riemann maps $\phi_{1} \colon Z_{1} \rightarrow \Omega_{1}$, $\phi_{2} \colon Z_2 \rightarrow \Omega_{2}$ onto the complementary components of a Jordan curve $\mathcal{C}$ with $g = \phi_{2}^{-1} \circ \phi_{1}|_{ \mathbb{S}^{1} }$; with the Carathéodory theorem we can make sense of the composition $\phi_{2}^{-1} \circ \phi_{1}|_{ \mathbb{S}^{1} }$ \cite{Gar:Mar:05}. Any such $g$ is called a \emph{welding homeomorphism} and $\mathcal{C}$ a \emph{welding curve}. A long-standing problem is to understand which homeomorphisms $g$ satisfy $\phi_{2}^{-1} \circ \phi_{1}|_{ \mathbb{S}^{1} }$ for some Riemann maps. We refer to the survey articles \cite{Ham:02}, \cite{You:15} for further background information.
    
    We also investigate the properties of $\widetilde{Z}$, given an arbitrary welding homeomorphism $g$. We show in \Cref{sec:harm:weld} that the $1$-dimensional Hausdorff measures on the seam $Q(S_{Z})$ and on (the tangents of) $\mathcal{C}$ are closely connected, using results from classical complex analysis \cite{Gar:Mar:05}. For example, our results show that a given subarc of the welding curve has tangents only in a set negligible to the $1$-dimensional Hausdorff measure if and only if the quotient map $Q$ collapses the corresponding part of the seam to a point.

    We present in Sections \ref{sec:nonexample1} and \ref{sec:nonexample2} examples illustrating that for some homeomorphisms $g$, after removing a portion $E'$ of the seam $Q( S_{Z} )$, one can find a $1$-quasiconformal embedding $\psi \colon \widetilde{Z} \setminus E' \rightarrow \mathbb{S}^{2}$, but not necessarily a quasiconformal homeomorphism $\Psi \colon \widetilde{Z} \rightarrow \mathbb{S}^{2}$. A similar phenomenom was investigated in \cite{Ham:91} and \cite{Bis:07:welding} in more detail.

    We now state our first result.
\begin{thm}\label{thm:mass:accumulation}
    Let $g \colon \mathbb{S}^{1} \rightarrow \mathbb{S}^{1}$ be an orientation-preserving homeomorphism. The following are quantitatively equivalent.
    \begin{enumerate}
        \item $g$ is $L$-bi-Lipschitz;
        \item there exists an $L'$-bi-Lipschitz homeomorphism $\Psi \colon \widetilde{Z} \rightarrow \mathbb{S}^{2}$;
        \item there exists $C' \geq 0$ such that for every $y \in Q(S_{Z})$,
        \begin{equation*}
            \liminf_{ r \rightarrow 0^{+} }
            \frac{ \mathcal{H}^{2}_{\widetilde{Z}}( \overline{B}_{\widetilde{Z}}( y, r ) ) }{ \pi r^2 }
            \leq
            C'.
        \end{equation*}
    \end{enumerate}
    In the implications "(1) $\Rightarrow$ (2)" we may take $L' = L$, in "(2) $\Rightarrow$ (3)" $C' = (L')^{4}$, and in "(3) $\Rightarrow$ (1)" $L = \pi C'$.
\end{thm}
    We prove "(1) $\Rightarrow$ (2)" by observing that if $g \colon \mathbb{S}^{1} \rightarrow \mathbb{S}^{1}$ admits an $L'$-bi-Lipschitz extension $\phi \colon \overline{Z}_{2} \rightarrow \overline{Z}_{2}$, the space $\widetilde{Z}$ has an $L'$-bi-Lipschitz parametrization. That we may take $L' = L$ in "(1) $\Rightarrow (2)$", follows by applying a known planar extension result \cite{Kal:14} and stereographic projection.
    
    The claim "(2) $\Rightarrow$ (3)" is a straightforward consequence of the properties of Hausdorff measures. The implication "(3) $\Rightarrow$ (1)" is proved by carefully analyzing the behaviour of the inclusion mappings $\iota_{i} \colon \overline{Z}_{i} \rightarrow \widetilde{Z}$ at the equator $\mathbb{S}^{1}$. Notice that the $\iota_{i}$ are $1$-Lipschitz everywhere and local isometries outside the equator. This implies $C' \geq 1$ in (3). \Cref{rem:sharpness} shows two ways to improve the bi-Lipschitz constant $\pi C'$. The improvements imply that as $C' \rightarrow 1^{+}$ in (3), the bi-Lipschitz constant of $g$ converges to one. In particular, (3) holds with $C' = 1$ if and only if $g$ is an isometry.
    
    \Cref{thm:mass:accumulation} is closely related to the following result.
\begin{thm}\label{cor:QC:Jacobianproblem}
    If an orientation-preserving homeomorphism $g \colon \mathbb{S}^{1} \rightarrow \mathbb{S}^{1}$ is $L$-bi-Lipschitz, there exists a $1$-quasiconformal homeomorphism $\varphi \colon \mathbb{S}^{2} \rightarrow \widetilde{Z}$ and a $K$-quasiconformal homeomorphism $h \colon \mathbb{S}^{2} \rightarrow \mathbb{S}^{2}$ such that the Jacobians satisfy
    \begin{equation}
        \label{eq:jacobian:comparable:cor}
        C^{-1}
        J_{ h }(x)
        \leq
        J_{ \varphi }(x)
        \leq
        C
        J_{ h }(x)
        \quad\text{ for $\mathcal{H}^{2}_{ \mathbb{S}^{2} }$-a.e. $x \in \mathbb{S}^{2}$}
    \end{equation}
    for $K = L^{4}$ and $C = L^{2}$. Conversely, if there exists $K$, $C$, and $h$ for which \eqref{eq:jacobian:comparable:cor} holds, then $g$ is $\pi ( K C )^2$-bi-Lipschitz.
\end{thm}
    The Jacobians are defined in \Cref{sec:measure}. We note that if $h \colon \mathbb{S}^{2} \rightarrow \mathbb{S}^{2}$ is an orientation-preserving quasiconformal homeomorphism, the $J_{h}$ coincides with the usual distributional Jacobian; see for example \cite[Section 3.8]{Ast:Iwa:Mar:09}.
    
    If $g$ is $L$-bi-Lipschitz, the existence of $\varphi$ and $h$ is a straightforward consequence of the implication "(1) $\Rightarrow$ (2)" in \Cref{thm:mass:accumulation}. If $h$ and $\varphi$ exist, we first check that $\Psi = h \circ \varphi^{-1}$ is bi-Lipschitz and use the implications "(2) $\Rightarrow$ (3) $\Rightarrow$ (1)" from \Cref{thm:mass:accumulation} to verify that $g$ is bi-Lipschitz.
    
    \Cref{cor:QC:Jacobianproblem} is a special case of the \emph{quasiconformal Jacobian problem}: which weights $\omega \colon \mathbb{S}^{2} \rightarrow \left[0, \infty\right]$ are comparable to the Jacobians of quasiconformal homeomorphisms $h \colon \mathbb{S}^{2} \rightarrow \mathbb{S}^{2}$; see \cite{Bo:Hei:Saks:04}, \cite{Bis:07}, and references therein for further reading.
    
    Given that (1) and (3) are equivalent in \Cref{thm:mass:accumulation}, it is not entirely clear for which classes of homeomorphisms one can expect $\widetilde{Z}$ to be quasiconformally equivalent to $\mathbb{S}^{2}$, or what kind of geometric properties one can expect from such a $\widetilde{Z}$.

\begin{ques}\label{Q1:welding}
    Let $\widetilde{Z}$ be the metric space obtained from a homeomorphism $g \colon \mathbb{S}^{1} \rightarrow \mathbb{S}^{1}$. When can we find a quasiconformal homeomorphism $\psi \colon \widetilde{Z} \rightarrow \mathbb{S}^{2}$? What kind of restrictions does this impose on $g$?
\end{ques}
    
    As an example, if $g$ is a welding homeomorphism corresponding to the von Koch snowflake, then $d_{Z}(x,y) = 0$ for every pair of points in the seam, see \Cref{rem:collapsing}. Hence $\widetilde{Z}$ can fail to be quasiconformally equivalent, or homeomorphic, to $\mathbb{S}^{2}$ when $g$ is a quasisymmetry. We show that a simple measure-theoretic assumption removes this obstruction.
\begin{prop}\label{thm:QS:abs}
    Let $g \colon \mathbb{S}^{1} \rightarrow \mathbb{S}^{1}$ be a quasisymmetry whose inverse is absolutely continuous. Then $\widetilde{Z}$ is quasiconformally equivalent to $\mathbb{S}^{2}$.
\end{prop}
    The absolute continuity of $g^{-1}$ is used in two ways. First, it guarantees that $\widetilde{Z} = ( Z, d_{Z} )$. Second, if $\psi \colon \overline{Z}_{2} \rightarrow \overline{Z}_{2}$ is a quasisymmetric extension of $g$, we show that the homeomorphism $H \colon \mathbb{S}^{2} \rightarrow \widetilde{Z}$ satisfying $H|_{Z_1} = \iota_{1}$ and $H|_{ Z_{2} } = \iota_{2} \circ \psi|_{Z_{2}}$ is quasiconformal. A key step in the proof is showing the Sobolev regularity $H^{-1} \in N^{1,2}( \widetilde{Z}, \mathbb{S}^{2} )$; the absolute continuity of $g^{-1}$ is applied here.
    
    \Cref{thm:QS:abs} is a special case of the following stronger result.
\begin{thm}\label{thm:biconformal:abs}
    Let $g \colon \mathbb{S}^{1} \rightarrow \mathbb{S}^{1}$ be an orientation-preserving homeomorphism whose inverse is absolutely continuous. If $g$ extends to a homeomorphism $\psi \colon \overline{Z}_{2} \rightarrow \overline{Z}_{2}$ for which $\psi|_{ Z_{2} }$ has exponentially integrable distortion, then $\widetilde{Z}$ is quasiconformally equivalent to $\mathbb{S}^{2}$.
\end{thm}
    We now explain the main steps of the proof of \Cref{thm:biconformal:abs}. We first show that there exists a homeomorphism $H \colon \mathbb{S}^{2} \rightarrow \widetilde{Z}$ with exponentially integrable distortion. We also have $H^{-1} \in N^{1,2}( \widetilde{Z}, \mathbb{S}^{2} )$; see \Cref{rem:regularity}. The exponential integrability of distortion of $H$ is used to verify the reciprocality condition of $\widetilde{Z}$, see \Cref{defi:reciprocal}. Then \cite[Theorem 1.4]{Raj:17} shows that $\widetilde{Z}$ is quasiconformally equivalent to $\mathbb{S}^{2}$. The key ingredients in the proof are the condenser estimates for mappings of exponentially distortion \cite{Ko:On:06}, applicable because $H^{-1} \in N^{1,2}( \widetilde{Z}, \mathbb{S}^{2} )$, and the Stoilow factorization theorem \cite[Chapter 20]{Ast:Iwa:Mar:09}. There are some known criteria which guarantee that $g$ admits an extension as in \Cref{thm:biconformal:abs}; see \cite{Za:08} \cite{Kar:Nta:21}.
    
    In \Cref{sec:nonexample1}, we present an example of $g \colon \mathbb{S}^{1} \rightarrow \mathbb{S}^{1}$ that is locally bi-Lipschitz outside a single point, but for which $\widetilde{Z}$ is not quasiconformally equivalent to $\mathbb{S}^{2}$. This illustrates that the absolute continuity of $g^{-1}$ is not enough to guarantee that $\widetilde{Z}$ is quasiconformally equivalent to $\mathbb{S}^{2}$. This fact is a consequence of the following result, partially answering \Cref{Q1:welding}.
\begin{thm}\label{thm:welding:positive}
    Suppose that $g \colon \mathbb{S}^{1} \rightarrow \mathbb{S}^{1}$ is an orientation-preserving homeomorphism for which there exists a quasiconformal homeomorphism $h \colon \mathbb{S}^{2} \rightarrow \widetilde{Z}$. Then $\widetilde{Z} = ( Z, d_{Z} )$ and there exists a $1$-quasiconformal homeomorphism $\pi \colon \mathbb{S}^{2} \rightarrow \widetilde{Z}$. Furthermore, $g$ is a welding homeomorphism whose welding curves are conformally removable.
\end{thm}
    The first step in the proof of \Cref{thm:welding:positive} is showing that $h$ can be assumed to be $1$-quasiconformal. Then, up to an orientation-reversing Möbius transformation, $\phi_{i} = h^{-1} \circ \iota_{i} \colon Z_{i} \rightarrow \mathbb{S}^{2}$ are Riemann maps with welding curve $\mathcal{C} = h^{-1}( Q( S_{Z} ) )$ and welding homeomorphism $\phi_{2}^{-1} \circ \phi_{1}|_{ \mathbb{S}^{1} }$. The equality $\widetilde{Z} = ( Z, d_{Z} )$ and the conformal removability of $\mathcal{C}$ follow from a connection we show between the tangents of the welding curve $\mathcal{C}$ and the Hausdorff $1$-measure on the seam $Q( S_{Z} )$; see \Cref{sec:harm:weld}. The equality $\widetilde{Z} = ( Z, d_{Z} )$ implies $g = \phi_{2}^{-1} \circ \phi_{1}|_{ \mathbb{S}^{1} }$.

    We recall that a compact proper subset $K \subset \mathbb{S}^{2}$ is \emph{conformally removable} if every homeomorphism $M \colon \mathbb{S}^{2} \rightarrow \mathbb{S}^{2}$ conformal in $\mathbb{S}^{2} \setminus K$ is Möbius. The von Koch snowflake example illustrates that conformal removability of a welding curve $\mathcal{C}$ is not enough to guarantee even that $\widetilde{Z}$ is homeomorphic to $\mathbb{S}^{2}$. We refer the reader to \cite{You:15} and \cite{You:18} for further reading on conformal weldings and the connections to conformal removability. See \cite{Her:Kos:03} for some results in the context of \Cref{thm:biconformal:abs}.
    
    The paper is structured as follows. In \Cref{sec:QC}, we introduce our notations and some preliminary results. In \Cref{sec:sew}, we analyze the distance $d_{Z}$ induced by any given homeomorphism $g \colon \mathbb{S}^{1} \rightarrow \mathbb{S}^{1}$. When $g$ is a welding homeomorphim, we establish in \Cref{sec:harm:weld} a connection between the geometry of the seam $S_{Z}$ and the tangents of the corresponding welding curves $\mathcal{C}$. We also prove \Cref{thm:welding:positive} in this section. In \Cref{sec:Ahlfors}, we prove Theorems \ref{thm:mass:accumulation} and \ref{cor:QC:Jacobianproblem}. \Cref{thm:QS:abs} and \Cref{thm:biconformal:abs} are proved in \Cref{sec:finite:distortion}. In \Cref{sec:obstructions}, we give some concluding remarks.
\section{Preliminaries}\label{sec:QC}

\subsection{Notation}\label{sec:notation} 

    Let $(Y,d_{Y})$ be a metric space. We sometimes drop the subscript from $d_Y$ when there is no chance for confusion. For all $Q \geq 0$, the \emph{$Q$-dimensional Hausdorff measure}, or a \emph{Hausdorff $Q$-measure}, is defined by
\[
	    \mathcal{H}_Y^{Q}(B)
	    =
	    \frac{\alpha(Q)}{2^Q}
	    \sup_{ \delta > 0 }
	    \inf
	        \left\{
	            \sum_{i=1}^\infty (\diam B_i)^Q
	            \colon
	            B \subset \bigcup_{i=1}^\infty B_i, 
	            \diam B_i < \delta
	        \right\} 
\]
	for all sets $B \subset Y$, where $\alpha(Q)$ is chosen so that $\mathcal{H}^{n}_{\mathbb{R}^{n}}$ coincides with the Lebesgue measure $\mathcal{L}^{n}$ for all positive integers.
    
    The \emph{length} of a path $\gamma\colon [a,b] \to Y$ is defined as 
\[ 
        \ell_d(\gamma) = \sup \sum_{i=1}^n d(\gamma(t_{i}), \gamma(t_{i+1})),
\]
    the supremum taken over all finite partitions $a = t_1 \leq t_2 \leq \cdots \leq t_{n+1} = b$. A path is \emph{rectifiable} if it has finite length.
    
    The \emph{metric speed} of a path $\gamma \colon \left[a,  b\right] \rightarrow Y$ at the point $t \in \left[a,  b\right]$ is defined as
\begin{equation*}
    \label{eq:metric:speed:path}
        v_{\gamma}(t)
        =
        \lim_{ t \neq s \rightarrow t }
            \frac{ d( \gamma( s ),  \gamma( t ) ) }{ |s-t| }
\end{equation*}
    whenever this limit exists. The limit exists $\mathcal{L}^{1}$-almost everywhere for every rectifiable path \cite[Theorem 2.1]{Dud:07}.

    A rectifiable path $\gamma\colon [a,b] \to Y$ is \emph{absolutely continuous} if for all $a \leq s \leq t \leq b$,
\begin{equation*}
    \label{eq:AC:characterization}
        d( \gamma(t),  \gamma(s) )
        \leq
        \int_{ s }^{ t }
            v_{ \gamma }( u )
        \,d\mathcal{L}^{1}( u )
\end{equation*}
    with $v_{ \gamma } \in L^{1}( \left[a,  b\right] )$ and $\mathcal{L}^{1}$ the Lebesgue measure on the real line. Equivalently, the rectifiable path $\gamma$ is absolutely continuous if it maps sets of $\mathcal{L}^{1}$-measure zero to sets of $\mathcal{H}_Y^{1}$-measure zero \cite[Section 3]{Dud:07}.
    
    Let $\gamma \colon \left[a, b\right] \rightarrow X$ be an absolutely continuous path. Then the \emph{(path) integral} of a Borel function $\rho \colon X \rightarrow \left[0,  \infty\right]$ \emph{over $\gamma$} is
\begin{equation}
    \label{eq:path:integral}
        \int_{ \gamma }
            \rho
        \,ds
        =
        \int_{ a }^{ b }
            ( \rho \circ \gamma )
            v_{ \gamma }
        \,d\mathcal{L}^{1}.
\end{equation}
    If $\gamma$ is rectifiable, then the path integral of $\rho$ over $\gamma$ is defined to be the path integral of $\rho$ over the arc length parametrization $\gamma_{s}$ of $\gamma$; see \cite[Chapter 5]{HKST:15} for further details.
    
    Given a Borel set $A \subset Y$, the \emph{length} of a path $\gamma \colon \left[a, b\right] \rightarrow Y$ \emph{in $A$} is defined as $\int_{ Y } \chi_{A}(y) \#( \gamma^{-1}(y) ) \,d\mathcal{H}^{1}_{Y}(y)$, where $\#( \gamma^{-1}(x) )$ is the counting measure of $\gamma^{-1}(x)$. For $A = Y$, \cite[Theorem 2.10.13]{Fed:69} states
    \begin{equation}
        \label{eq:areaformula:federer}
        \ell( \gamma )
        =
        \int_{ Y }
            \#( \gamma^{-1}(y) )
        \,d\mathcal{H}^{1}_{Y}(y).
    \end{equation}
    When $\gamma$ is rectifiable, for every Borel function $\rho \colon Y \rightarrow \left[0, \infty\right]$,
    \begin{equation}
        \label{eq:areaformula}
        \int_{ \gamma } \rho \,ds
        =
        \int_{ Y }
            \rho( y )
            \#( \gamma^{-1}(y) )
        \,d\mathcal{H}^{1}_{Y}(y).
    \end{equation}
    The equality \eqref{eq:areaformula} follows from \cite[Theorem 2.10.13]{Fed:69} via a standard approximation argument using simple functions.
    
\subsection{Metric Sobolev spaces}\label{sec:sobolev}

    In this section we give an overview of Sobolev theory in the metric surface setting, and refer to \cite{HKST:15} for a comprehensive introduction.

    Let $\Gamma$ be a family of paths in $Y$. A Borel function $\rho\colon Y \to [0, \infty]$ is \emph{admissible} for $\Gamma$ if the path integral $\int_\gamma \rho\,ds \geq 1$ for all rectifiable paths $\gamma \in \Gamma$. Given $1 \leq p < \infty$, the \emph{$p$-modulus} of $\Gamma$ is
        \[\Mod_{p} \Gamma = \inf \int_Y \rho^p\,d\mathcal{H}_Y^2,\]
    where the infimum is taken over all admissible functions $\rho$. Observe that if $\Gamma_{1}$ and $\Gamma_{2}$ are path families and every path $\gamma_{1} \in \Gamma_{1}$ contains a subpath $\gamma_{2} \in \Gamma_{2}$, then $\Mod_{p} \Gamma_{1} \leq \Mod_{p} \Gamma_{2}$. In particular, this holds if $\Gamma_{1} \subset \Gamma_{2}$. When $p = 2$, and there is no chance for confusion, we omit the subscript from $\Mod_{2}$.
    
    If $\rho$ is admissible for a path family $\Gamma \setminus \Gamma_{0}$, where $\Mod_{p} \Gamma_{0} = 0$, we say that $\rho$ is \emph{$p$-weakly admissible} for $\Gamma$. If a property holds for every path $\gamma \in \Gamma$ except in a subfamily of $p$-modulus zero, the property is said to hold \emph{on $p$-almost every path} in $\Gamma$. We also refer to $2$-almost every path as \emph{almost every} path.
    
    We recall the following lemma \cite[Lemma 5.2.8]{HKST:15}.
    \begin{lemm}\label{lemm:negligible}
    Let $1 \leq p < \infty$. A family of nonconstant paths $\Gamma$ satisfies $\Mod_{p} \Gamma = 0$ if and only if there exists $\rho \colon Y \rightarrow \left[0, \infty\right]$, $\rho \in L^{p}( Y )$ with
    \begin{equation*}
        \infty = \int_{ \gamma } \rho \,ds \quad \text{for every $\gamma \in \Gamma$.}
    \end{equation*}
    \end{lemm}
    Let $\psi \colon (Y,d_Y) \to (Z,d_Z)$ be a mapping between metric spaces $Y$ and $Z$. A Borel function $\rho \colon Y \to [0, \infty]$ is an \emph{upper gradient of $\psi$} if
    \[
        d_Y(\psi(x),\psi(y)) \leq \int_\gamma \rho\,ds
    \]
    for every rectifiable path $\gamma\colon [a,b] \to Y$ connecting $x$ to $y$. The function $\rho$ is a \emph{$p$-weak upper gradient of $\psi$} if the same holds for $p$-almost every rectifiable path.
    
    A $p$-weak upper gradient $\rho \in L^{p}_{\loc}( Y )$  of $\psi$ is \emph{minimal} if it satisfies $\rho \leq \widetilde{\rho}$ almost everywhere for all $p$-weak upper gradients $\widetilde{\rho} \in L^{p}_{\loc}( Y )$ of $\psi$. If $\psi$ has a $p$-weak upper gradient $\rho \in L^{p}_{\loc}( Y )$, then $\psi$ has a minimal $p$-weak upper gradient, which we denote by $\rho_\psi$. We refer to Section 6 of \cite{HKST:15} and Section 3 of \cite{Wil:12} for further details. Minimal $2$-weak upper gradients are also refered to as \emph{minimal weak upper gradients}.

    Fix a point $z \in Z$, and let $d_z = d_Z(\cdot,z)$. The space $L^{p}( Y,  Z )$ is defined as the collection of measurable maps $\psi \colon Y \to Z$ such that $d_z \circ \psi$ is in $L^p(Y)$. Moreover, $L^{p}_{\loc}( Y,  Z )$ is defined as those measurable maps $\psi \colon Y \to Z$ for which, for all $y \in Y$, there is an open set $U \subset Y$ containing $y$ such that $\psi|_U$ is in $L^p(U,Z)$.
    
    The metric Sobolev space $N^{1, p}_{\loc}( Y,  Z )$ consists of those maps $\psi \colon Y \rightarrow Z$ in $L^{p}_{\loc}( Y,  Z )$ that have a minimal $p$-weak upper gradient $\rho_{ \psi } \in L^{p}_{\loc}( Y )$.
    
    For subsets $\emptyset \neq U \subset Y$, we say that $\psi \in N^{1,  p}( U,  Z )$ if $\psi|_{U} \in N^{1,  p}_{\loc}( U, Z )$, $\rho_{\psi|U} \in L^{p}( U )$ and $\psi|_{U} \in L^{p}( U, Z )$. If $Z = \mathbb{R}$, we denote $N^{1,p}( U, Z ) = N^{1,p}( U )$, and in the case $p = 2$,
    \begin{equation*}
        E( \psi ) \coloneqq 2^{-1} \norm{ \rho_{ \psi } }_{ L^{2}( U ) }^{2}.
    \end{equation*}
    We refer to $E( \psi )$ as the \emph{Dirichlet energy} of $\psi$.

    We repeatedly use the following technical lemma in later sections.
    \begin{lemm}\label{lemm:techincal:lemma}
    Let $\psi \colon Y \rightarrow Z$ be continuous, $\rho \colon Y \rightarrow \left[0, \infty\right]$ a Borel function and $\gamma \colon \left[0, 1\right] \rightarrow Y$ absolutely continuous with $\int_{ \gamma } \rho \,ds < \infty$.
    
    Suppose that $E \subset Y$ is compact, $\mathcal{H}^{1}_{Z}( \psi(E) ) = 0$, and $\ell( \psi \circ \gamma|_{I} ) \leq \int_{ \gamma|_{I} } \rho \,ds$ for each closed interval $I \subset \left[0, 1\right] \setminus \gamma^{-1}( E )$. Then $\ell( \psi \circ \gamma ) \leq \int_{ \gamma } \rho \,ds$.
    \end{lemm}
    \begin{proof}
    First, for every closed interval $J \subset \left[0, 1\right] \setminus \gamma^{-1}( E )$, $\psi \circ \gamma|_{J}$ is absolutely continuous with $v_{ \psi \circ \gamma }(s) \leq ( \rho \circ \gamma )(s) v_{ \gamma }(s)$ for $\mathcal{L}^{1}$-almost every $s \in J$. This follows from \cite[Proposition 6.3.2]{HKST:15}.
    
    Second, consider the connected components $\left\{ I_{i} \right\}_{ i = 1 }^{ \infty }$ of $\left[0, 1\right] \setminus ( \psi \circ \gamma )^{-1}( \psi(E) )$. Notice that $I_{i} \subset \left[0, 1\right] \setminus \gamma^{-1}(E)$ for every $i$.
    
    Let $J_{i} = \overline{ I_{i} }$. Then $v_{ \psi \circ \gamma }(s) \leq ( \rho \circ \gamma|_{ J_{i} } )( s )$ $\mathcal{L}^{1}$-almost everywhere on $J_{i}$ (on $I_{i}$). This fact, the continuity of $\psi \circ \gamma$ and $\int_{ \gamma } \rho \,ds < \infty$ imply
    \begin{equation*}
        \ell( \psi \circ \gamma|_{ J_{i} } )
        \leq
        \int_{ I_{i} }
            ( \rho \circ \gamma ) v_{ \gamma }
        \,ds
        <
        \infty.
    \end{equation*}
    By summing over $i$, we conclude
    \begin{equation*}
        \sum_{ i = 1 }^{ \infty }
        \ell( \psi \circ \gamma|_{ J_{i} } )
        \leq
        \int_{ \bigcup_{ i = 1 }^{ \infty } I_{i} }
            ( \rho \circ \gamma )v_{ \gamma }
        \,ds
        \leq
        \int_{ \gamma }
            \rho
        \,ds.
    \end{equation*}
    Given $\mathcal{H}^{1}_{Z}( \psi(E) ) = 0$, \eqref{eq:areaformula:federer} and \eqref{eq:areaformula} imply
    \begin{align*}
        \ell( \psi \circ \gamma )
        &=
        \int_{ Z \setminus \psi(E) }
            \#( ( \psi \circ \gamma )^{-1}(x) )
        \,d\mathcal{H}^{1}_{Z}(x)
        \\
        &\leq
        \sum_{ i = 1 }^{ \infty }
        \int_{ Z \setminus \psi(E) }
            \#( ( \psi \circ \gamma|_{ J_{i} } )^{-1}(x) )
        \,d\mathcal{H}^{1}_{Z}(x)
        \\
        &=
        \sum_{ i = 1 }^{ \infty } \ell( \psi \circ \gamma|_{ J_{i} } )
        \leq
        \int_{ \gamma }
            \rho
        \,ds.
    \end{align*}
    Hence $\ell( \psi \circ \gamma ) \leq \int_{ \gamma } \rho \,ds$.
    \end{proof}    
    
    \subsection{Measure theory}\label{sec:measure}
    Let $Y$ be a Borel subset of a complete and separable metric space. A Borel measure $\mu$ on $Y$ is \emph{$\sigma$-finite} if there exists a Borel decomposition $\left\{ B_{i} \right\}_{ i = 1 }^{ \infty }$ of $Y$ for which $\mu( B_{i} ) < \infty$ for every $i$.
    
    A pair of $\sigma$-finite Borel measures $\mu$ and $\nu$ on $Y$ are said to be \emph{mutually singular} if there exists a Borel set $B \subset Y$ such that $\mu( B ) = 0$ and $\nu( Y \setminus B ) = 0$. The measure $\mu$ admits a \emph{Lebesgue decomposition} (with respect to $\nu$), where $\mu = f \cdot \nu + \mu^{\perp}$, with $\mu^{\perp}$ and $\nu$ mutually singular and $f$ Borel measurable \cite[Sections 3.1-3.2 in Volume I]{Bog:07}. We say that $\mu$ and $\nu$ are \emph{mutually absolutely continuous} if $\mu = f \cdot \nu$ with density $f > 0$ $\nu$-almost everywhere.
    
    Given a homeomorphism $\psi \colon Y \rightarrow Z$ and measures $\nu$ on $Y$ and $\mu$ on $Z$, the measure $\psi^{*}\mu( B ) = \mu( \psi(B) )$ is called the \emph{pullback measure}. Such a measure admits a decomposition $\psi^{*}\mu = f \cdot \nu + \mu^{\perp}$ with $\nu$ and $\mu^{\perp}$ mutually singular. If $\nu = \mathcal{H}^{2}_{Y}$ and $\mu = \mathcal{H}^{2}_{Z}$, the density $f$ is called the \emph{Jacobian} of $\psi$ and denoted by $J_{\psi}$.

\subsection{Quasiconformal mappings} \label{sec:quasiconformal_mappings}
    Here we define quasiconformal maps and recall some basic facts.
\begin{defi}\label{def:QCmaps}
    Let $( Y, d_{Y} )$ and $( Z, d_{Z} )$ be metric spaces with locally finite Hausdorff $2$-measures. A homeomorphism $\psi \colon ( Y, d_{Y} ) \rightarrow ( Z, d_{Z} )$ is \emph{quasiconformal} if there exists $K \geq 1$ such that for all path families $\Gamma$ in $Y$
\begin{equation}
        \label{eq:def:QCmaps}
        K^{-1} \Mod \Gamma
        \leq
        \Mod \psi \Gamma
        \leq
        K \Mod \Gamma,
\end{equation}
    where $\psi \Gamma = \left\{ \psi \circ \gamma \colon \gamma \in \Gamma \right\}$. If \eqref{eq:def:QCmaps} holds with a constant $K \geq 1$, we say that $\psi$ is \emph{$K$-quasiconformal}.
\end{defi}
    A special case of \cite[Theorem 1.1]{Wil:12} yields the following.
\begin{thm} \label{prop:williams:L-Wversion}
    Let $Y$ and $Z$ be locally compact separable metric spaces with locally finite Hausdorff $2$-measure and $\psi \colon Y \rightarrow Z$ a homeomorphism. The following are equivalent for the same constant $K > 0$:
\begin{enumerate}
    \item[(i)] $\Mod \Gamma \leq K \Mod \psi\Gamma$ for all path families $\Gamma$ in $Y$.
    \item[(ii)] $\psi \in N_{\loc}^{1,2}(Y,Z)$ and satisfies
        \[ \rho_\psi^2(y) \leq KJ_\psi(y) \]
    for $\mathcal{H}_Y^2$-almost every $y \in Y$.
\end{enumerate}
\end{thm}
    The \emph{outer dilatation} of $\psi$ is the smallest constant $K_{O} \geq 0$ for which the modulus inequality $\Mod \Gamma \leq K_{O} \Mod \psi \Gamma$ holds for all $\Gamma$ in $Y$. The \emph{inner dilatation} of $\psi$ is the smallest constant $K_{I} \geq 0$ for which $\Mod \psi \Gamma \leq K_{I} \Mod \Gamma$ holds for all $\Gamma$ in $Y$. The number $K( \psi ) = \max\left\{ K_{I}( \psi ), K_{O}( \psi ) \right\}$ is the \emph{maximal dilatation} of $\psi$.
    
    For a set $G \subset Y$ and disjoint sets $F_1, F_2 \subset G$, let $\Gamma(F_1,F_2; G)$ denote the family of paths with each path starting at $F_{1}$, ending at $F_{2}$ and whose images are contained in $G$. A \emph{quadrilateral} is a set $Q$ homeomorphic to $[0,1]^2$ with boundary $\partial Q$ consisting of four boundary arcs, overlapping only at the end points, labelled $\xi_1, \xi_2, \xi_3, \xi_4$ in cyclic order.
    
    A \emph{metric surface} is a separable metric space $Y$ with locally finite Hausdorff $2$-measure that is homeomorphic to a (connected) $2$-manifold without boundary.
\begin{defi}\label{defi:reciprocal}
    A metric surface $Y$ is \emph{reciprocal} if there exists a constant $\kappa  \geq 1 $ such that
\begin{align}
	\label{upper:bound}
	\kappa^{-1}
	\leq
	\Mod \Gamma\left( \xi_{1},  \xi_{3};  Q \right)
	\Mod \Gamma\left( \xi_{2},  \xi_{4};  Q \right)
	\leq
	\kappa
\end{align}
    for every quadrilateral $Q \subset Y$, and 
\begin{equation}
	\label{point:zero:modulus}
	\lim_{ r \rightarrow 0^{+} }
	\Mod \Gamma\left( \overline{B}_{Y}( y,  r ),  Y \setminus B_{Y}( y,  R );  \overline{B}_{Y}( y,  R ) \right)
    =
    0
\end{equation}
    for all $y \in Y$ and $R > 0$ such that $Y \setminus B_Y( y, R ) \neq \emptyset$.
\end{defi}
    We note that for every metric surface,
\begin{equation}
    \label{eq:lower:bound}
    \kappa_{0}^{-1}
    \leq
	\Mod \Gamma\left( \xi_{1},  \xi_{3};  Q \right)
	\Mod \Gamma\left( \xi_{2},  \xi_{4};  Q \right),
\end{equation}
    with $\kappa_{0} = ( 4 / \pi )^2$ \cite{EB:PC:21} \cite{RR:19}.
    
\begin{thm}[Theorem 1.4 of \cite{Raj:17}]\label{thm:raj}
    Let $( Y, d_{Y} )$ be a metric surface homeomorphic to $\mathbb{R}^{2}$ or to $\mathbb{S}^{2}$. Then there exists a quasiconformal embedding $\psi \colon ( Y, d_{Y} ) \rightarrow \mathbb{S}^{2}$ if and only if $Y$ is reciprocal.
\end{thm}
    Theorem 1.3 of \cite{Iko:19} shows that if a metric surface $( Y, d_{Y} )$ can be covered by quasiconformal images of domains $V \subset \mathbb{R}^{2}$, then $( Y, d_{Y} )$ is quasiconformally equivalent to a Riemannian surface. In particular, we have the following.
\begin{thm}\label{thm:iko}
    Let $( Y, d_{Y} )$ be a metric surface homeomorphic to $\mathbb{S}^{2}$. Then there exists a quasiconformal homeomorphism $\psi \colon ( Y, d_{Y} ) \rightarrow \mathbb{S}^{2}$ if and only if each point $y \in Y$ is contained in an open set $U$ from which there exists a quasiconformal homeomorphism $\phi \colon U \rightarrow V \subset \mathbb{R}^{2}$.
\end{thm}
    Since \eqref{eq:lower:bound} holds, we have the following result.
\begin{prop}[Corollary 12.3 of \cite{Raj:17}]\label{prop:outer:to:maximal}
    Let $Y$ be a metric surface, $U \subset Y$ a domain, and $\psi \colon U \rightarrow \Omega \subset \mathbb{R}^{2}$ a homeomorphism. If $K_{O}( \psi ) < \infty$, then $\psi$ is $K$-quasiconformal for $K = \left( 2 \cdot \kappa_{0} \right) \cdot K_{O}( \psi )$.
\end{prop}

    \section{Hemispheres}\label{sec:sew}
    We construct a (pseudo)distance $d_{Z}$ on $Z$ using a \emph{predistance} $D \colon Z \times Z \rightarrow \left[0, \infty\right]$ defined in the following way, with the identification $S_{Z} \subset \overline{Z}_{1}$ for the seam,
    \begin{equation*}
        D( x, y )
        =
        \left\{
        \begin{split}
            &\infty,
            &&\text{ if } (x, y) \in Z_{1} \times Z_{2} \cup Z_{2} \times Z_{1},
            \\
            &\min\left\{
                \sigma( x, y ),
                \sigma( g(x), g(y) )
            \right\},
            &&\text{ if } x, y \in S_{Z},
            \\
            &\sigma( x, y ),
            &&\text{ otherwise}.
        \end{split}
        \right.
    \end{equation*}
    Then we denote $d_{Z}( x, y ) = \inf \sum_{ i = 1 }^{ n } D( x_{i}, x_{i+1} )$, the infimum taken over finite chains $( x_{i} )_{ i = 1 }^{ n+1 }$ for which $x_{1} = x$ and $x_{n+1} = y$. We obtain a metric space $\widetilde{Z}$ and a quotient map $Q \colon Z \rightarrow \widetilde{Z}$ by identifying $( x, y ) \in Z \times Z$ whenever $d_{Z}( x, y ) = 0$, and setting $d_{\widetilde{Z}}( x, y ) = d_{Z}( Q^{-1}(x), Q^{-1}(y) )$ for each $x, y \in \widetilde{Z}$.
    
    In this section, we focus on analyzing the distance $d_{Z}$ on the seam $S_{Z}$. The main results of this section are Lemmas \ref{lemm:jumping:overtheseam} and \ref{lemm:inclusion} and \Cref{lemm:hausdorff}.
    
    In the following two lemmas we abuse notation and identify $\iota_{i}( Z_{i} )$ with $Z_{i}$ when convenient.
    \begin{lemm}\label{lemm:jumping}
    The following hold:
    \begin{enumerate}
        \item Let $x, y \in \mathbb{S}^{1} \subset \overline{Z}_{1}$ and $( x_{i} )_{ i = 1 }^{ n + 1 }$ a chain with $x_{1} = x$, $x_{n+1} = y$, and $x_{i} \in Z_{1}$ otherwise. Then $\sum_{ i = 1 }^{ n } D( x_{i},x_{i+1} ) \geq D( x, y )$.
        \item Let $x, y \in \mathbb{S}^{1} \subset \overline{Z}_{1}$ and $( x_{i} )_{ i = 1 }^{ n + 1 }$ a chain with $g(x_{1}) = g(x)$, $g(x_{n+1}) = g(y)$, and $x_{i} \in Z_{2}$ otherwise. Then $\sum_{ i = 1 }^{ n } D( x_{i},x_{i+1} ) \geq D( x, y )$.
    \end{enumerate}
    \end{lemm}
    \begin{proof}
    Given the chain from the claim (1), for every $i$, $D( x_{i}, x_{i+1} ) = \sigma( x_{i}, x_{i+1} )$. Thus, $\sum_{ i = 1 }^{ n } D( x_{i},x_{i+1} ) \geq \sigma( x_{1}, x_{n+1} ) \geq D( x_{1}, x_{n+1} )$. The corresponding inequalities hold for the chain from (2).
    \end{proof}
    \Cref{lemm:jumping} implies that when computing $d_{Z}(\iota_1(x), \iota_1(y))$ for $x, y \in \mathbb{S}^{1}$, it is sufficient to consider chains with intermediate points staying within the seam.
    \begin{lemm}\label{lemm:jumping:overtheseam}
    If $x, y \in Z_{1}$, then
    \begin{equation}
        \label{eq:energyminimal:chain}
        d_{Z}( \iota_{1}(x), \iota_{1}(y) )
        =
        \left\{
        \begin{split}
            &\sigma(x,y),
            \quad\text{or there exist $w, w' \in \mathbb{S}^{1}$ with}
            \\
            &\sigma( x, w )
            +
            d_{Z}( \iota_{1}(w), \iota_{1}(w') )
            +
            \sigma( w', y )
            \leq
            \sigma(x,y).
        \end{split}
        \right.
    \end{equation}
    The corresponding identity holds for points $x, y \in Z_{2}$.
    
    Furthermore, if $x \in Z_{1}$ and $y \in Z_{2}$, there exist $w, w' \in \mathbb{S}^{1}$ such that
    \begin{equation}
        \label{eq:energyminimal:chain:seam}
        d_{Z}( \iota_{1}(x), \iota_{2}(y) )
        =
        \sigma( x, w )
        +
        d_{Z}( \iota_{1}(w), \iota_{1}(w') )
        +
        \sigma( g(w'), y ).
    \end{equation}
    \end{lemm}
    \begin{proof}
    We show \eqref{eq:energyminimal:chain}. Suppose that there exists a sequence $\epsilon_{j} \rightarrow 0^{+}$ and a sequence of chains $( x_{i,j} )_{ i = 1 }^{ n_{j} + 1 }$ joining $x$ to $y$ with $d_{Z}( \iota_{1}(x), \iota_{1}(y) ) \geq - \epsilon_{j} + \sum_{ i = 1 }^{ n_{j} } D( x_{i,j}, x_{i+1,j} )$ so that every chain has an element in $\mathbb{S}^{1}$. 
    If $i_1$ is the first index for which $x_{i,j} \in \mathbb{S}^{1}$ and $i_{2}$ the last one, then
    \begin{align*}
        \sum_{ i = 1 }^{ n_{j} }
            D( x_{i,j}, x_{i+1,j} )
        &\geq
        \sigma( x, x_{i_{1},j} )
        +
        d_{Z}( \iota_{1}( x_{i_{1},j} ), \iota_{1}( x_{i_{2},j} ) )
        +
        \sigma( x_{ i_{2}, j }, y )
        \\
        &\geq
        \inf\left\{ \sigma( x, w ) + d_{Z}( \iota_{1}(w), \iota_{1}(w') ) + \sigma( w', y ) \right\},
    \end{align*}
    the infimum taken over every $w, w' \in \mathbb{S}^{1}$. Observe that the infimum is realized by some $w, w' \in \mathbb{S}^{1}$. Given such $w, w' \in \mathbb{S}^{1}$, we pass to the limit $j \rightarrow \infty$ and conclude
    \begin{equation*}
        d_{Z}( \iota_{1}(x), \iota_{1}(y) )
        \geq
        \sigma( x, w ) + d_{Z}( \iota_{1}(w), \iota_{1}(w') ) + \sigma( w', y ).
    \end{equation*}
    Since "$\leq$" holds for every pair $w, w' \in \mathbb{S}^{1}$, the lower equality in \eqref{eq:energyminimal:chain} follows.
    
    If no such sequence of $\epsilon_{j} \rightarrow 0^{+}$ exists, then there exists $\epsilon_{0} > 0$ such that for every $\epsilon_{0} > \epsilon > 0$, any chain joining $x$ to $y$ with $d_{Z}( \iota_{1}(x), \iota_{2}(y) ) \geq - \epsilon + \sum_{ i = 1 }^{ n } D( x_{i}, x_{i+1} )$ does not intersect $\mathbb{S}^{1}$. Hence $\sum_{ i = 1 }^{ n } D( x_{i}, x_{i+1} ) \geq \sigma( x, y )$. So, either way, we obtain \eqref{eq:energyminimal:chain}. The claims for each $x, y \in Z_{2}$ and $(x,y) \in Z_{1} \times Z_{2}$ are proved in a similar manner.
    \end{proof}
    
    For $i = 1,2$, we denote $\widetilde{\iota}_{i} \coloneqq Q \circ \iota_{i} \colon \overline{Z}_{i} \rightarrow \widetilde{Z}$. \Cref{lemm:jumping:overtheseam} implies that $\widetilde{\iota}_{i}$ is $1$-Lipschitz everywhere and a local isometry in $Z_{i}$. We also establish that $\widetilde{\iota}_{i}$ is \emph{monotone}, i.e, the preimage of a point is a compact and conected set.
    \begin{lemm}\label{lemm:inclusion}
    For $i = 1,2$, the inclusion map $\widetilde{\iota}_{i} \colon \overline{Z}_{i} \rightarrow \widetilde{Z}$ is $1$-Lipschitz everywhere and a local isometry on $Z_{i}$. Moreover, for every $z \in \widetilde{Z}$, the preimage $\widetilde{\iota}_{i}^{-1}(z)$ is compact and connected. It contains two or more points only if $\widetilde{\iota}_{i}^{-1}(z) \subset \mathbb{S}^{1}$.
    \end{lemm}
    Before proving \Cref{lemm:inclusion}, we show two auxiliary results.
    
    \begin{lemm}\label{lemm:partitioning}
    Let $x, y \in \mathbb{S}^{1}$ be distinct. Then there exists an arc $\gamma \colon \left[0,1\right] \rightarrow \mathbb{S}^{1}$ joining $x$ to $y$ with $D( \iota_{1}(x), \iota_{1}(y) ) = \min\left\{ \ell( \gamma ), \ell( g \circ \gamma ) \right\}$. The arc satisfies
    \begin{equation*}
        D( \iota_{1}(x), \iota_{1}(y) )
        \geq
        \sup_{ \left\{ t_{i} \right\}_{ i = 1 }^{ n+1 } }
        \sum_{ i = 1 }^{ n }
        D( \iota_{1}( \gamma(t_{i}) ), \iota_{1}( \gamma(t_{i+1}) ) ),
    \end{equation*}
    the supremum taken over finite partitions of $\left[0, 1\right]$. In particular, $D( \iota_{1}(x), \iota_{1}(y) ) \geq \ell( \widetilde{\iota}_{1}( \gamma ) )$.
    \end{lemm}
    \begin{proof}
    The existence of $\gamma$ with $D( \iota_{1}(x), \iota_{1}(y) ) = \min\left\{ \ell( \gamma ), \ell( g \circ \gamma ) \right\}$ follows from the fact that $\sigma$ is geodesic on $\mathbb{S}^{1}$. We identify $\iota_{1}( x )$ with $x$ for every $x \in \mathbb{S}^{1}$ in the following computations.
    
    The claim about the partitions is a consequence of the following observation and induction: If $0 \leq a < s < b \leq 1$, then
    \begin{equation}
        \label{eq:lemm:partitioning}
        D( \gamma(a), \gamma(b) )
        \geq
        D( \gamma(a), \gamma(s) )
        +
        D( \gamma(s), \gamma(b) ).
    \end{equation}
    We first assume that $D( \gamma(a), \gamma(b) ) = \sigma( \gamma(a), \gamma(b) )$. Then $\gamma$ is a length-minimizing geodesic joining $\gamma(a)$ and $\gamma(b)$. Consequently,
    \begin{equation*}
        \sigma( \gamma(a), \gamma(b) )
        =
        \sigma( \gamma(a), \gamma(s) ) + \sigma( \gamma(s), \gamma(b) ).
    \end{equation*}
    Since $\sigma( c, d ) \geq D( c, d )$ holds for every $c, d \in \mathbb{S}^{1}$, the inequality \eqref{eq:lemm:partitioning} holds in this case. In the remaining case, $g \circ \gamma$ is a length-minimizing geodesic joining $g( \gamma(a) )$ and $g( \gamma(b) )$ and
    \begin{equation*}
        \sigma( g( \gamma(a) ), g( \gamma(b) ) )
        =
        \sigma( g( \gamma(a) ), g( \gamma(s) ) )
        +
        \sigma( g( \gamma(s) ), g( \gamma(b) ) ).
    \end{equation*}
    Since $\sigma( g(c), g(d) ) \geq D( c, d )$ for every $c, d \in \mathbb{S}^{1}$, the inequality \eqref{eq:lemm:partitioning} holds also in this case.
    
    The partition claim implies $D( x, y ) \geq \sum_{ i = 1 }^{ n } d_{ \widetilde{Z} }( \widetilde{\iota}_{1}( \gamma( t_{i} ) ), \widetilde{\iota}_{1}( \gamma( t_{i+1} ) ) )$ for every partition $\left\{ t_{i} \right\}_{ i = 1 }^{ n+1 }$ of $\left[0, 1\right]$. The inequality $D( x, y ) \geq \ell( \widetilde{\iota}_{1}( \gamma ))$ follows by taking the supremum over such partitions.
    \end{proof}

    \begin{lemm}\label{lemm:partitioning:adv}
    Let $x, y \in \mathbb{S}^{1}$ be distinct. Then there exists an arc $\gamma \colon \left[0, 1\right] \rightarrow \mathbb{S}^{1}$ joining $x$ to $y$ such that $d_{ \widetilde{Z} }(  \widetilde{\iota}_{1}( x ), \widetilde{\iota}_{1}( y ) ) = \ell( \widetilde{\iota}_{1}( \gamma ) )$.
    \end{lemm}
    \begin{proof}
    Let $\epsilon > 0$. The defining property of $d_{Z}$ and \Cref{lemm:jumping} imply the existence of a chain $\left\{ x_{i} \right\}_{ i = 1 }^{ n + 1 } \subset \mathbb{S}^{1}$ joining $x$ to $y$ for which
    \begin{equation*}
        d_{Z}( \iota_{1}(x), \iota_{1}(y) )
        \geq
        - \epsilon
        +
        \sum_{ i = 1 }^{ n }
        D( \iota_{1}( x_{i} ), \iota_{1}( x_{i+1} ) ).
    \end{equation*}
    For each $i$, \Cref{lemm:partitioning} yields the existence of an arc $\theta_{i} \colon \left[0,1\right] \rightarrow \mathbb{S}^{1}$ joining $x_{i}$ to $x_{i+1}$ with $D( \iota_{1}( x_{i} ), \iota_{1}( x_{i+1} ) ) \geq \ell( \widetilde{\iota}_{1}( \theta_{i} ) )$. Let $\theta$ denote the concatenation of these paths. Then $d_{Z}( \iota_{1}(x), \iota_{1}(y) ) \geq - \epsilon + \ell( \widetilde{\iota}_{1}( \theta ) )$.
    
    Let $\theta' \colon \left[0, 1\right] \rightarrow \mathbb{S}^{1}$ be an arc joining $x$ to $y$ within the image of $\theta$. Applying \eqref{eq:areaformula} on $\widetilde{Z}$ with $\rho \equiv \chi_{ \widetilde{Z} }$ implies that $\ell( \widetilde{\iota}_{1}( \theta ) ) \geq \ell( \widetilde{\iota}_{1}( \theta' ) )$. Such a $\theta'$ is one of the arcs joining $x$ to $y$ within $\mathbb{S}^{1}$.
    
    Let $\epsilon_{j} \rightarrow 0^{+}$ and consider $\theta'_j$ as above for every such $\epsilon_{j}$. Up to passing to a subsequence and relabeling, we may assume that every such $\theta'_j$ is the same arc $\theta'$. Passing to the limit $j \rightarrow \infty$ establishes $d_{Z}( \iota_{1}( x ), \iota_{1}(y) ) \geq \ell( \widetilde{\iota}_{1}( \theta' ) ) \geq d_{Z}( \iota_{1}(x), \iota_{1}(y) )$. We set $\gamma = \theta'$ to conclude the proof.
    \end{proof}

    \begin{proof}[Proof of \Cref{lemm:inclusion}]
    The claimed $1$-Lipschitz and local isometry properties follow of $\widetilde{\iota}_{1}$ from \Cref{lemm:jumping:overtheseam}. The local isometry property implies that given $z \in \widetilde{Z}$, the preimage $\widetilde{\iota}_{1}^{-1}( z )$ has more than two points only if the preimage is a subset of $\mathbb{S}^{1}$.
    
    Suppose the existence of a distinct pair $x, y \in \widetilde{\iota}_{1}^{-1}(z)$. Then $x, y \in \mathbb{S}^{1}$. \Cref{lemm:partitioning:adv} shows that there exists an arc $\gamma$ joining $x$ to $y$ within $\mathbb{S}^{1}$ satisfying $$0 = d_{ \widetilde{Z} }( \widetilde{\iota}_{1}(x), \widetilde{\iota}_{1}(y) ) = \ell( \widetilde{\iota}_{1}( \gamma ) ).$$ This implies $|\gamma| \subset \widetilde{\iota}_{1}^{-1}( z )$. Since $x$ and $y$ were arbitrary, we conclude that $\widetilde{\iota}_{1}^{-1}( z )$ is path connected. Consequently, $\widetilde{\iota}_{1}^{-1}( z )$ is a connected and compact subset of $\mathbb{S}^{1}$.
    
    The properties of $\widetilde{\iota}_{2}$ follow from a symmetry in the argument. Hence the claim follows.
    \end{proof}
    
    \begin{prop}\label{lemm:hausdorff}
    Let $g \colon ( \mathbb{S}^{1}, \mathcal{H}^{1}_{ \mathbb{S}^{1} } ) \rightarrow ( \mathbb{S}^{1}, \mathcal{H}^{1}_{ \mathbb{S}^{1} } )$ be a homeomorphism with $g^{*} \mathcal{H}^{1}_{ \mathbb{S}^{1} } = v_{g}\mathcal{H}^{1}_{ \mathbb{S}^{1} } + \mu^{ \perp }$ with $\mathcal{H}^{1}_{ \mathbb{S}^{1} }$ and $\mu^{\perp}$ mutually singular. Then, for every Borel set $B \subset \mathbb{S}^{1}$,
    \begin{equation}
        \label{eq:identity:Z}
        \mathcal{H}^{1}_{d_{\widetilde{Z}}}( \widetilde{\iota}_{1}(B) )
        =
        \int_{B}
            \min\left\{ 1, v_{g} \right\}
        \,d\mathcal{H}^{1}_{ \mathbb{S}^{1} }
        =
        \int_{ \widetilde{\iota}_{1}( B ) }
            {\#( \widetilde{\iota}_{1}^{-1}(z) )}
        \,d\mathcal{H}^{1}_{ \widetilde{Z} }(z).
    \end{equation}
    Moreover, for every $x, y \in \mathbb{S}^{1}$, there exists an arc $|\gamma| \subset \mathbb{S}^{1}$ joining $x$ to $y$ for which
    \begin{equation}
        \label{eq:lengthdistance}
        d_{ \widetilde{Z} }( \widetilde{\iota}_{1}(x), \widetilde{\iota}_{1}(y) )
        =
        \ell( \widetilde{\iota}_{1}( \gamma ) ).
    \end{equation}
    \end{prop}
    
    Before proving \Cref{lemm:hausdorff}, we first consider a Carathéodory construction on $\mathbb{S}^{1}$. First, fix a Borel set $B_{0} \subset \mathbb{S}^{1}$ for which $\mathcal{H}^{1}_{ \mathbb{S}^{1} }( B_{0} ) = 0$ and $\mu^{\perp}( \mathbb{S}^{1} \setminus B_{0} ) = 0$. Set $\nu^{ABS}( B ) \coloneqq \int_{ B } \min\left\{ 1, v_{g} \right\} \chi_{ \mathbb{S}^{1} \setminus B_0 } \,d\mathcal{H}^{1}_{ \mathbb{S}^{1} }$ for all Borel sets $B \subset \mathbb{S}^{1}$.
    
    For every arc $\gamma \colon \left[0, 1\right] \rightarrow \mathbb{S}^{1}$, we denote $\xi^{ABS}( |\gamma| ) \coloneqq \nu^{ABS}( |\gamma| )$ and $\xi( |\gamma| ) \coloneqq D( \gamma(0), \gamma(1) )$. The set function $\xi^{ABS}$ and the family of arcs $|\gamma| \subset \mathbb{S}^{1}$ yields Carathéodory premeasures $\nu^{ABS}_{\delta}$ for each $\delta > 0$.
     
    \begin{lemm}\label{lemm:equality}
    For every Borel set $B \subset \mathbb{S}^{1}$, we have $\nu^{ABS}(B) = \sup_{ \delta > 0 } \nu^{ABS}_{\delta}(B) \geq \mathcal{H}^{1}_{ \widetilde{Z} }( \widetilde{\iota}_{1}(B) )$.
    \end{lemm}
    \begin{proof}
    The equality $\nu^{ABS}(B) = \sup_{ \delta > 0 } \nu^{ABS}_{\delta}(B)$ follows from the fact that $\nu^{ABS}$ is a finite Borel measure.
    
    We denote $B_{1} = \left\{ v_{g} \geq 1 \right\} \cup B_{0}$ and $B_{2} = \mathbb{S}^{1} \setminus B_{1}$. If $B \subset \mathbb{S}^{1}$ is Borel, we have
    \begin{equation*}
        \mathcal{H}^{1}_{ \widetilde{Z} }( \widetilde{\iota}_{1}(B) )
        =
        \sum_{ i = 1 }^{2}
        \mathcal{H}^{1}_{ \widetilde{Z} }( \widetilde{\iota}_{1}(B \cap B_{i}) )
        \leq
        \mathcal{H}^{1}_{ \mathbb{S}^{1} }( B \cap B_{1} )
        +
        \mathcal{H}^{1}_{ \mathbb{S}^{1} }( g(B) \cap g(B_{2}) )
    \end{equation*}
    since $\widetilde{\iota}_{i}$ is $1$-Lipschitz for $i = 1,2$. The right-hand side equals $\nu^{ABS}( B )$. Therefore $\mathcal{H}^{1}_{ \widetilde{Z} }( \widetilde{\iota}_{1}( B ) ) \leq \nu^{ABS}( B )$ holds for all Borel sets.
    \end{proof}
    
    \begin{lemm}\label{lemm:equality:lowerbound}
    Let $x, y \in \mathbb{S}^{1}$ be distinct and $\gamma \colon \left[0, 1\right] \rightarrow \mathbb{S}^{1}$ an arc joining $x$ to $y$ such that $d_{ \widetilde{Z} }(  \widetilde{\iota}_{1}( x ), \widetilde{\iota}_{1}( y ) ) = \ell( \widetilde{\iota}_{1}( \gamma ) )$. Then $\mathcal{H}^{1}_{ \widetilde{Z} }( \widetilde{\iota}_{1}(|\gamma|) ) = d_{ \widetilde{Z} }(  \widetilde{\iota}_{1}( x ), \widetilde{\iota}_{1}( y ) ) = \nu^{ABS}( |\gamma| )$.
    \end{lemm}
    \begin{proof}
    Let $\pi/2 > \delta_{0} > 0$ be such that
    \begin{equation*}
        D( \iota_{1}(a), \iota_{1}(b) ) < \delta_{0}
        \quad\text{implies}\quad
        \max\left\{ \sigma(a,b), \sigma( g(a), g(b) ) \right\} < \pi/2.
    \end{equation*}
    Given such a pair $a, b \in \mathbb{S}^{1}$, the length-minimizing geodesic $\theta \colon \left[0, 1\right] \rightarrow \mathbb{S}^{1}$ joining $a$ to $b$ satisfies $\xi( |\theta| ) = \min\left\{ \ell( \theta ), \ell( g \circ \theta ) \right\}$. Then $\xi( |\theta| ) \geq \xi^{ABS}( |\theta| )$.
    
    Let $\gamma$ be as in the claim. Let $0 < \delta < \delta_{0}$ and $0 < \epsilon < \delta/2$. We consider a partition $\left\{ t_{i} \right\}_{ i = 1 }^{ n +1 }$ of $\left[0, 1\right]$ such that $\sigma( \gamma( t_{i} ), \gamma(t_{i+1}) ) < \delta/2$ for every $i$. Then there exists a chain $\left\{ x_{i,j} \right\}_{ j = 1 }^{ n_{i}+1 } \subset \mathbb{S}^{1}$ joining the ends of $\gamma|_{ \left[t_{i}, t_{i+1}\right] }$ so that
    \begin{equation*}
        d_{Z}( \iota_{1}\circ \gamma( t_{i} ), \iota_{1}\circ \gamma( t_{i+1} ) )
        \geq
        -
        \frac{ \epsilon }{ n }
        +
        \sum_{ j = 1 }^{ n_{i} }
        D( \iota_{1}( x_{i,j} ), \iota_{1}( x_{i,j+1} ) ).
    \end{equation*}
    In particular, $D( \iota_{1}( x_{i,j} ), \iota_{1}( x_{i,j+1} ) ) < \delta < \delta_{0}$ for every $j$. Hence the length-minimizing geodesic $\gamma_{i,j}$ joining $x_{i,j}$ to $x_{i,j+1}$ satisfies the assumptions of \Cref{lemm:partitioning}. For every $i$, \Cref{lemm:partitioning} implies that, up to further partitioning the paths $\gamma_{i,j}$ and relabeling, we may assume $\sigma( x_{i,j}, x_{i,j} ) < \delta$ for every $j$. Given this property, we conclude $D( \iota_{1}( x_{i,j} ), \iota_{1}( x_{i,j+1} ) ) = \xi( |\gamma_{i,j} | ) \geq \xi^{ABS}( |\gamma_{i,j}| )$ and
    \begin{equation*}
        \ell( \widetilde{\iota}_{1}( \gamma ) )
        =
        \sum_{ i = 1 }^{ n }
        d_{Z}( \iota_{1}\circ \gamma( t_{i} ), \iota_{1}\circ \gamma( t_{i+1} ) )
        \geq
        -
        \epsilon
        +
        \nu_{ \delta }^{ABS}\left( \bigcup_{ i = 1 }^{ n } \bigcup_{ j = 1 }^{ n_{i} }|\gamma_{i,j}| \right).
    \end{equation*}
    Since the concatenation $\theta_{i}$ of $\left\{ \gamma_{i,j} \right\}_{ j = 1 }^{ n_{i} }$ is a path joining $\gamma( t_{i} )$ to $\gamma( t_{i+1} )$, the concatenation $\theta$ of $\left\{ \theta_{i} \right\}_{ i = 1 }^{ n }$ is a path joining $x$ to $y$. Hence $\bigcup_{ i = 1 }^{ n } \bigcup_{ j = 1 }^{ n_{i} }|\gamma_{i,j}| = |\theta|$ contains $|\gamma|$ or $\mathbb{S}^{1} \setminus |\gamma|$, and
    \begin{equation*}
        \ell( \widetilde{\iota}_{1}( \gamma ) )
        \geq
        -
        \epsilon
        +
        \min\left\{
            \nu_{ \delta}^{ABS}( |\gamma| ),
            \nu_{ \delta }^{ABS}( \mathbb{S}^{1} \setminus |\gamma| )
        \right\}.
    \end{equation*}
    After passing to $\epsilon \rightarrow 0^{+}$ and then to $\delta \rightarrow 0^{+}$, we conclude
    \begin{align*}
        \mathcal{H}^{1}_{ \widetilde{Z} }( \widetilde{\iota}_{1}( \gamma ) )
        =
        \ell( \widetilde{\iota}_{1}( \gamma ) )
        \geq
        \min\left\{
            \nu^{ABS}( |\gamma| ),
            \nu^{ABS}( \mathbb{S}^{1} \setminus |\gamma| )
        \right\}.
    \end{align*}
    If we had $\nu^{ABS}( |\gamma| ) > \nu^{ABS}( \mathbb{S}^{1} \setminus |\gamma| )$, this would contradict \Cref{lemm:equality} and the length-minimizing property of $\widetilde{\iota}_{1}( \gamma )$. Hence $\nu^{ABS}( |\gamma| ) \leq \nu^{ABS}( \mathbb{S}^{1} \setminus |\gamma| )$, and $\mathcal{H}^{1}_{ \widetilde{Z} }( \widetilde{\iota}_{1}(|\gamma|) ) = d_{ \widetilde{Z} }(  \widetilde{\iota}_{1}( x ), \widetilde{\iota}_{1}( y ) ) = \nu^{ABS}( |\gamma| )$ follows from \Cref{lemm:equality}.
    \end{proof}
    \begin{proof}[Proof of \Cref{lemm:hausdorff}]
    The existence of $\gamma$ and equality in \eqref{eq:lengthdistance} already follows from \Cref{lemm:partitioning}.
    
    We claim that \eqref{eq:identity:Z} holds. To this end, we consider three arcs $\gamma_{i} \colon \left[0, 1\right] \rightarrow \mathbb{S}^{1}$ overlapping only at their end points, whose images cover $\mathbb{S}^{1}$, with the arcs satisfying $\nu^{ABS}( |\gamma_i| ) \leq \nu^{ABS}( \mathbb{S}^{1} \setminus |\gamma_{i}| )$.
    
    Lemmas \ref{lemm:equality} and \ref{lemm:equality:lowerbound} imply that $\widetilde{\iota}_{1} \circ \gamma_i$ is a length-minimizing geodesic joining its end points and $\nu^{ABS}( |\gamma_i| ) = \mathcal{H}^{1}_{ \widetilde{Z} }( \widetilde{\iota}_{1}( |\gamma_i| ) )$. Lemma \ref{lemm:equality} implies that the metric speed of $\widetilde{\iota}_{1}|_{ |\gamma_{i}| }$ is bounded from above by $\min\left\{1,v_{g}\right\}$. Hence the equality $\nu^{ABS}( |\gamma_i| ) = \mathcal{H}^{1}_{ \widetilde{Z} }( \widetilde{\iota}_{1}( |\gamma_i| ) )$ forces the metric speed of $\widetilde{\iota}_{1}$ to equal $\min\left\{ 1, v_{g} \right\}$ $\mathcal{H}^{1}_{ \mathbb{S}^{1} }$-almost everywhere on $|\gamma_{i}|$ for $i = 1,2,3$. The equality \eqref{eq:identity:Z} follows from the area formula \eqref{eq:areaformula} and the fact that $\#( \widetilde{\iota}_{1}^{-1}(x) ) = 1$ $\mathcal{H}^{1}_{ \widetilde{Z} }$-almost everywhere. The fact $\#( \widetilde{\iota}_{1}^{-1}(x) ) = 1$ $\mathcal{H}^{1}_{ \widetilde{Z} }$-almost everywhere follows from the monotonicity of $\widetilde{\iota}_{1}$ and the integrability of the multiplicity. The integrability of the multiplicity follows from \eqref{eq:areaformula:federer}.
    \end{proof}

    \begin{rem}\label{rem:Ah:Beu}
    We consider a $2\pi$-periodic doubling measure $\mu$ on $\mathbb{R}$ with $2 \pi = \mu( \left[0, 2\pi\right] )$ such that for some Borel set $B \subset \left[0, 2\pi\right]$, $\mathcal{L}^{1}( B ) = 0 = \mu( \left[0, 2\pi\right] \setminus B )$, the existence of which is established by Ahlfors--Beurling \cite[Section 7]{Ah:Beu:56}. Then $\psi(x) = \int_{0}^{x} \,d\mu$ is a homeomorphism and there exists a quasisymmetry $g \colon \mathbb{S}^{1} \rightarrow \mathbb{S}^{1}$ with $\theta \circ \psi = g \circ \theta$, where $\theta(t) = ( \cos(t), \sin(t), 0 )$. Then $v_{g}$ in \eqref{eq:identity:Z} is identically zero. Consequently, $d_{Z} \equiv 0$ on the seam $S_{Z}$.
    \end{rem}

\section{Harmonic measure and welding homeomorphisms}\label{sec:harm:weld}
    We consider a welding homeomorphism $g \colon \mathbb{S}^{1} \rightarrow \mathbb{S}^{1}$ and a welding circle $\mathcal{C}$ with complementary components $\Omega_{1}$ and $\Omega_{2}$, Riemann maps $\phi_{i} \colon Z_{i} \rightarrow \Omega_{i}$ for $i = 1,2$, and $g = \phi_{2}^{-1} \circ \phi_{1}|_{ \mathbb{S}^{1} }$. In this section, we consider the \emph{harmonic measures} $\omega_{i}( E ) = \phi_{i}^{*}\mathcal{H}^{1}_{ \mathbb{S}^{1} }( E )/( 2 \pi )$ for all Borel sets $E \subset \mathbb{S}^{2}$.
    
    We define a homeomorphism $\pi \colon \mathbb{S}^{2} \rightarrow ( Z, d_{Z} )$ and a quotient map $\widetilde{\pi} \colon \mathbb{S}^{2} \rightarrow \widetilde{Z}$ via the formulas
    \begin{equation}
        \label{eq:quotientmap}
        \pi(x)
        =
        \left\{
        \begin{split}
            &\iota_{1} \circ \phi_{1}^{-1}(x),
            &&\quad \text{when } x \in \overline{ \Omega_{1} },
            \\
            &\iota_{2} \circ \phi_{2}^{-1}(x),
            &&\quad \text{when } x \in \Omega_{2}
        \end{split}
        \right.
        \quad\text{and}\quad
        \widetilde{\pi} = Q \circ \pi.
    \end{equation}
    Recall that $Q \colon Z \rightarrow \widetilde{Z}$ is the quotient map identifying $x, y \in Z$ whenever $d_{Z}( x, y ) = 0$. \Cref{lemm:inclusion} implies that $\widetilde{\pi}$ is monotone and $\widetilde{\pi}^{-1}( x )$ contains two or more points only if $x$ is a point of the seam $Q( S_{Z} )$, and in such a case $\widetilde{\pi}^{-1}(x) \subset \mathcal{C}$.
    
    For $\alpha = 1,2$, we denote, for every Borel set $B \subset \mathbb{S}^{2}$,
    \begin{equation}
        \label{eq:1D:hausdorffpullback}
        \widetilde{\pi}^{*}\mathcal{H}^{\alpha}_{ \widetilde{Z} }( B )
        \coloneqq
        \int_{ \widetilde{Z} }
            {\#}( B \cap \widetilde{\pi}^{-1}(x) )
        \,d\mathcal{H}^{\alpha}_{ \widetilde{Z} }(x)
        =
        \mathcal{H}^{\alpha}_{ \widetilde{Z} }( \widetilde{\pi}(B) ),
    \end{equation}
    where the multiplicity can be ignored in the case $\alpha = 2$ since it equals one outside the negligible set $Q( S_{Z} )$. For $\alpha = 1$, the multiplicity is two or more only when it is $\infty$ and this happens in a set of negligible $\mathcal{H}^{1}_{ \widetilde{Z} }$-measure. Either way, the multiplicity is negligible in \eqref{eq:1D:hausdorffpullback}, so the second equality is justified.
        
\begin{prop}\label{thm:failure:terrible}
    Let $g$ be a welding homeomorphism with a welding circle $\mathcal{C}$ and $I \subset \mathcal{C}$ a subarc. Then $d_{\widetilde{Z}}( \widetilde{\pi}(x), \widetilde{\pi}(y) ) = 0$ for all $x, y \in I$ if and only if $\omega_{1}|_{ I }$ and $\omega_{2}|_{ I }$ are mutually singular. If such an interval exists, then $\widetilde{Z}$ is not quasiconformally equivalent to $\mathbb{S}^{2}$.
\end{prop}

\begin{rem}\label{rem:collapsing}
    If $g$ is a welding homeomorphism obtained from \Cref{rem:Ah:Beu} or any welding $g$ corresponding to the von Koch snowflake \cite[Example 4.3]{Gar:Mar:05}, \Cref{thm:failure:terrible} implies that $Q( S_{Z} )$ is a singleton. In particular, $\widetilde{Z}$ is not even homeomorphic to the sphere. For a given $g$, this happens if and only if $g^{*}\mathcal{H}^{1}_{ \mathbb{S}^{1} }$ and $\mathcal{H}^{1}_{ \mathbb{S}^{1} }$ are mutually singular.
\end{rem}
    
    A key step in the proof of the conformal removability in \Cref{thm:welding:positive} is the following.
\begin{prop}\label{prop:quotientmap}
    Let $g$ be a welding homeomorphism and $\widetilde{\pi}$ as in \eqref{eq:quotientmap}. Then $\widetilde{\pi}$ is continuous, monotone, and surjective. Moreover, for all path families $\Gamma$ on $\mathbb{S}^{2}$, $\Mod \Gamma \leq \Mod \widetilde{\pi} \Gamma$. The metric space $\widetilde{Z}$ is quasiconformally equivalent to $\mathbb{S}^{2}$ if and only if $\widetilde{\pi}$ is a homeomorphism for which $\Mod \Gamma = \Mod \widetilde{\pi} \Gamma$ for all path families.
\end{prop}
    The proof of \Cref{prop:quotientmap} requires some preparatory work. Given the curve $\mathcal{C}$, we say that $x_{0} \in \mathcal{C}$ is a \emph{tangent point} if there exists a homeomorphism $\gamma \colon \left( -\epsilon, \epsilon \right) \rightarrow \mathcal{C}' \subset \mathcal{C}$ with $\gamma( 0 ) = x_{0}$, and a tangent vector $v_{0} \in T_{ x_{0} }\mathbb{S}^{2}$ with unit length such that for every smooth $f \colon \mathbb{S}^{2} \rightarrow \mathbb{R}$, its differential $df$ satisfies 
    \begin{equation*}
        df( v_{0} )
        =
        \lim_{ t \rightarrow 0^{+} }
            \frac{ f( \gamma( t ) ) - f( x_{0} ) }{ \sigma( \gamma(t), x_{0} ) }
        \quad\text{and}\quad
        df( -v_{0} )
        =
        \lim_{ t \rightarrow 0^{-} }
            \frac{ f( \gamma( t ) ) - f( x_{0} ) }{ \sigma( \gamma(t), x_{0} ) }.
    \end{equation*}
    If $v_{0}$ exists, the tangent vector $v_{0}$ is independent of the parametrization $\gamma$ and $\mathcal{C}'$ up to multiplication by $-1$; see \cite[Chapter II, Section 4]{Gar:Mar:05}. The collection of \emph{tangents points} of $\mathcal{C}$ is denoted by $\mathrm{Tn}( \mathcal{C} )$. The key properties of $\mathrm{Tn}( \mathcal{C} )$ are self-contained in the following statement.
    
    \begin{lemm}\label{lemm:harmonicmeasure:abs}
    The Borel set $\mathrm{Tn}( \mathcal{C} )$ has $\sigma$-finite Hausdorff $1$-measure. Moreover, on the set $\mathrm{Tn}( \mathcal{C} )$, the measures $\omega_{1}$, $\omega_{2}$, and $\mathcal{H}^{1}_{ \mathcal{C} }$ are mutually absolutely continuous.
    
    Given any Borel set $E \subset \mathcal{C}$ with $\omega_{1}( E ) \cdot \omega_{2}( E ) > 0$, the restrictions $\omega_{1}|_{ E }$ and $\omega_{2}|_{ E }$ are mutually singular on $E$ if and only if $\mathcal{H}^{1}_{ \mathcal{C} }( \mathrm{Tn}( \mathcal{C} ) \cap E ) = 0$.
    \end{lemm}
    \begin{proof}
    The Borel measurability of $\mathrm{Tn}( \mathcal{C} )$ follows from \cite[Chapter II, Theorem 4.2]{Gar:Mar:05} which connects the tangents of $\mathcal{C}$ and the angular derivatives of any given Riemann map $\phi_{1}' \colon Z_{1} \rightarrow \Omega_{1}$, where $\partial \Omega_{1} = \mathcal{C}$. The fact that $\mathrm{Tn}( \mathcal{C} )$ has $\sigma$-finite Hausdorff $1$-measure follows from \cite[Chapter VI, Theorem 4.2]{Gar:Mar:05}. 
    
    Theorem 6.3 of \cite[Chapter VI]{Gar:Mar:05} states that if a Borel set $E \subset \mathcal{C}$ is such that $\omega_{1}( E ) \cdot \omega_{2}( E ) > 0$, then $\omega_{1}|_{ E }$ and $\omega_{2}|_{ E }$ are mutually singular on $E$ if and only if $\mathcal{H}^{1}_{ \mathcal{C} }( \mathrm{Tn}( \mathcal{C} ) \cap E ) = 0$.
    
    The fact that on the set $\mathrm{Tn}( \mathcal{C} )$ the measures $\omega_{1}$, $\omega_{2}$, and $\mathcal{H}^{1}_{ \mathcal{C} }$ are mutually absolutely continuous follows from \cite[Chapter VI, Theorem 4.2 and the following discussion on p. 211]{Gar:Mar:05}.
    \end{proof}
    
    \begin{lemm}\label{lemm:mutualABS:tan}
    The measures $\chi_{ \mathcal{C} }\widetilde{\pi}^{*}\mathcal{H}^{1}_{ \widetilde{Z} }$, $\chi_{ \mathrm{Tn}( \mathcal{C} ) }\omega_{1}$, $\chi_{ \mathrm{Tn}( \mathcal{C} ) }\omega_{2}$ and $\chi_{ \mathrm{Tn}( \mathcal{C} ) }\mathcal{H}^{1}_{ \mathcal{C} }$ are mutually absolutely continuous.
    
    More precisely, a given Borel set $B \subset \mathrm{Tn}( \mathcal{C} )$ has positive $1$-dimensional Hausdorff measure if and only if $\mathcal{H}^{1}_{ \widetilde{Z} }( \widetilde{\pi}(B) ) > 0$. Furthermore, if $B \subset \mathcal{C} \setminus \mathrm{Tn}( \mathcal{C} )$, then $\mathcal{H}^{1}_{ \widetilde{Z} }( \widetilde{\pi}(B) ) = 0$.
    \end{lemm}
    \begin{proof}
    We write $g^{*}\mathcal{H}^{1}_{ \mathbb{S}^{1} } = v_{g}\mathcal{H}^{1}_{ \mathbb{S}^{1} } + 2 \pi \cdot \mu^{ \perp }$ with $\mathcal{H}^{1}_{ \mathbb{S}^{1} }$ and $\mu^{ \perp }$ are mutually singular. We recall from \Cref{lemm:hausdorff} that for every Borel set $B \subset \mathcal{C}$,
    \begin{equation}
        \label{eq:ABS}
        \mathcal{H}^{1}_{ \widetilde{Z} }( \widetilde{\pi}(B) )
        =
        \int_{ \phi_{1}^{-1}( B ) }
            \min\left\{ 1, v_{g} \right\}
        \,d\mathcal{H}^{1}_{ \mathbb{S}^{1} }.
    \end{equation}
    We denote $h = v_{g} \circ \phi_{1}^{-1}$ and observe the equality $\omega_{2} = h \omega_{1} + ( \phi_{1} )_{*}\mu^{\perp}$. Then \eqref{eq:ABS} is equivalent to
    \begin{equation}
        \label{eq:harmonicmeasure}
        ( 2 \pi )^{-1}
        \mathcal{H}^{1}_{ \widetilde{Z} }( \widetilde{\pi}(B) )
        =
        \int_{ B }
            \min\left\{ 1, h \right\}
        \,d\omega_{1}.
    \end{equation}
    \Cref{lemm:harmonicmeasure:abs} implies that the measures $\chi_{ \mathcal{C} \setminus \mathrm{Tn}( \mathcal{C} ) } \omega_{1}$ and $\chi_{\mathcal{C} \setminus \mathrm{Tn}( \mathcal{C} )} \omega_{2}$ are mutually singular. Consequently, $h = 0$ $\omega_{1}$-almost everywhere in $\mathcal{C} \setminus \mathrm{Tn}( \mathcal{C} )$. In particular, if $B = \mathcal{C} \setminus \mathrm{Tn}( \mathcal{C} )$, the left-hand side equals zero in \eqref{eq:harmonicmeasure}.
    
    \Cref{lemm:harmonicmeasure:abs} yields that the measures $\chi_{ \mathrm{Tn}( \mathcal{C} ) }\omega_{1}$, $\chi_{ \mathrm{Tn}( \mathcal{C} ) }\omega_{2}$ and $\chi_{ \mathrm{Tn}( \mathcal{C} ) }\mathcal{H}^{1}_{ \mathcal{C} }$ are mutually absolutely continuous. Hence $\infty > h > 0$ $\omega_{1}$-almost everywhere in $\mathrm{Tn}( \mathcal{C} )$. This implies that the measure in \eqref{eq:harmonicmeasure} is mutually absolutely continuous with the measures $\chi_{ \mathrm{Tn}( \mathcal{C} ) }\omega_{1}$, $\chi_{ \mathrm{Tn}( \mathcal{C} ) }\omega_{2}$ and $\chi_{ \mathrm{Tn}( \mathcal{C} ) }\mathcal{H}^{1}_{ \mathcal{C} }$. The claim follows from the equalities \eqref{eq:1D:hausdorffpullback} for $\alpha = 1$.
    \end{proof}
    
    \begin{proof}[Proof of \Cref{thm:failure:terrible}]
    Fix a subarc $I \subset \mathcal{C}$. \Cref{lemm:hausdorff} implies that $\widetilde{\pi}(I)$ has zero $\mathcal{H}^{1}_{  \widetilde{Z}  }$-measure if and only if for every $x, y \in I$, $d_{\widetilde{Z}}( \widetilde{\pi}(x), \widetilde{\pi}(y) ) = 0$ if and only if $v_{g} = 0$ $\mathcal{H}^{1}_{ \mathbb{S}^{1} }$-almost everywhere on $\phi_{1}^{-1}(I)$. Equivalently, $\omega_{1}|_{I}$ and $\omega_{2}|_{I}$ are mutually singular.
    
    \Cref{lemm:inclusion} shows that $\widetilde{Z} \neq ( Z, d_{Z} )$ if and only if there exists a closed arc $I \subset \mathbb{S}^{1}$ such that $y = \widetilde{\iota}_{1}( I )$. Assume that such an $I$ exists. Having fixed $x_{0} \in Z_{1}$ and $0 < s < \sigma( x_{0}, \mathbb{S}^{1} )$, there exists $c = c( x_{0}, I, s )$ for which
    \begin{equation*}
        \Mod \Gamma( I, \overline{B}_{ \mathbb{S}^{2} }( x_{0}, s ); I \cup Z_{1} )
        \geq
        c
        >
        0;
    \end{equation*}
    a positive lower bound can be shown, for example, by estimating the modulus of all geodesics joining $I$ to $\overline{B}_{ \mathbb{S}^{2} }( x_{0}, s )$ in $I \cup Z_{1}$.
    
    When $R > 0$ is small enough, for every $R > r > 0$ and every path in $\Gamma( I, \overline{B}_{ \mathbb{S}^{2} }( x_{0}, s ) ); I \cup Z_{1} )$, we find a subpath $\gamma' \colon \left[0, 1\right] \rightarrow Z_{1}$ so that $\widetilde{\iota}_{1} \circ \gamma$ joins $\overline{B}_{\widetilde{Z}}( y, r )$ to $\widetilde{Z} \setminus B_{\widetilde{Z}}( y,  R )$ within $\overline{B}_{\widetilde{Z}}( y,  R )$. Since $\widetilde{\iota}_{1}$ is a local isometry off the seam, this implies
    \begin{equation*}
        \liminf_{ r \rightarrow 0^{+} }
        \Mod \Gamma\left(
            \overline{B}_{\widetilde{Z}}( y,  r ), 
            \widetilde{Z} \setminus B_{\widetilde{Z}}( y,  R );
            \overline{B}_{\widetilde{Z}}( y,  R )
        \right)
        \geq
        c.
    \end{equation*}
    Recalling \Cref{thm:raj}, we see that $\widetilde{Z}$ is not quasiconformally equivalent to $\mathbb{S}^{2}$.
    \end{proof}
    
    \begin{lemm}\label{lemm:quotient:Sobolev}
    For $i = 1, 2$, let $\rho_{i} \colon \Omega_{i} \rightarrow \left[0, \infty\right]$ denote the operator norm of the differential of $D( \phi_{i}^{-1} )$. Then
    \begin{equation}
        \label{eq:weakuppergradient}
        G
        =
        \chi_{ \Omega_{1} }
        \rho_{1}
        +
        \chi_{ \Omega_{2} } \rho_{2}
        +
        \infty \cdot \chi_{ \mathrm{Tn}( \mathcal{C} ) }
        \in
        L^{2}( \mathbb{S}^{2} )
    \end{equation}
    is a weak upper gradient of $\widetilde{\pi}$.
    \end{lemm}

    \begin{proof}
    The $L^{2}$-integrability of $G$ follows from the change of variables formulas of the Riemann maps $\phi_{1}$ and $\phi_{2}$ and the fact that $\mathrm{Tn}( \mathcal{C} )$ has negligible area. Hence, as a consequence of \Cref{lemm:negligible}, $G$ is integrable along almost every absolutely continuous path $\gamma \colon \left[0, 1\right] \rightarrow \mathbb{S}^{2}$. Given such a $\gamma$, we claim that
    \begin{equation}
        \label{eq:uppergradient:inequa}
        d_{\widetilde{Z}}( \widetilde{\pi}( \gamma(0) ), \widetilde{\pi}( \gamma(1) ) )
        \leq
        \int_{ \gamma }
            G
        \,ds,
    \end{equation}
    implying that $G$ is a weak upper gradient of $\widetilde{\pi}$.

    Since $G$ is integrable along $\gamma$, $\gamma$ has negligible length in $\mathrm{Tn}( \mathcal{C} )$. Then \eqref{eq:areaformula} implies $\mathcal{H}^{1}_{ \mathbb{S}^{2} }( \mathrm{Tn}( \mathcal{C} ) \cap \abs{\gamma} ) = 0$. We conclude $\mathcal{H}^{1}_{ \widetilde{Z} }( \widetilde{\pi}( \mathcal{C} ) \cap | \widetilde{\pi} \circ \gamma | ) = 0$ from \Cref{lemm:mutualABS:tan}. The assumptions of \Cref{lemm:techincal:lemma} are satisfied and the conclusion $\ell( \widetilde{\pi} \circ \gamma ) \leq \int_{ \gamma } G \,ds$ follows. The inequality \eqref{eq:uppergradient:inequa} is a consequence.
    \end{proof}
    
    We define the \emph{Jacobian} of $\widetilde{\pi}$ to be the density of $\widetilde{\pi}^{*}\mathcal{H}^{2}_{ \widetilde{Z} }$, defined in \eqref{eq:1D:hausdorffpullback}, with respect to $\mathcal{H}^{2}_{ \mathbb{S}^{2} }$.
    \begin{lemm}\label{lemm:pullbackmeasure:modulus}
    The mapping $\widetilde{\pi}$ satisfies Lusin's Condition ($N$) and the Jacobian $J_{ \widetilde{\pi} }$ coincides with $G^{2}$ $\mathcal{H}^{2}_{ \mathbb{S}^{2} }$-almost everywhere, with $G$ being from \eqref{eq:weakuppergradient}.
    \end{lemm}
    \begin{proof}
    The Lusin's Condition ($N$) of $\widetilde{\pi}$ follows from the fact that $\widetilde{\pi}( \mathcal{C} )$ has negligible $\mathcal{H}^{2}_{ \widetilde{Z} }$-measure, the fact that $\iota_{i} \colon Z_{i} \rightarrow \widetilde{Z}_{i}$ is a local isometry, and as $\phi_{i}^{-1} \colon \Omega_{i} \rightarrow Z_{i}$ satisfies Condition ($N$). Here $J_{ \widetilde{\pi} } = 0$ $\mathcal{H}^{2}_{\mathbb{S}^{2}}$-almost everywhere on $\mathcal{C}$, so the equality $J_{ \widetilde{\pi} } = G^{2}$ follows from the fact that $\phi_{1}$ and $\phi_{2}$ are Riemann maps.
    \end{proof}
    
    \begin{proof}[Proof of \Cref{prop:quotientmap}]
    The claimed topological properties of $\widetilde{\pi}$ were already verified at the beginning of this section. Lemmas \ref{lemm:quotient:Sobolev} and \ref{lemm:pullbackmeasure:modulus} prove that $J_{ \widetilde{\pi} } = G^{2} \in L^{1}( \mathbb{S}^{2} )$ with $G$ being a weak upper gradient of $\widetilde{\pi}$. This fact and the fact that the multiplicity of $\widetilde{\pi}$ is negligible for $\widetilde{\pi}^{*}\mathcal{H}^{2}_{ \widetilde{Z} }$ imply $\Mod \Gamma \leq \Mod \widetilde{\pi} \Gamma$ for all path families $\Gamma$.
    
    Lastly, we argue that a $K$-quasiconformal map $\psi \colon \widetilde{Z} \rightarrow \mathbb{S}^{2}$ exists (for some $K \geq 1$) if and only if $\widetilde{\pi}$ is a $1$-quasiconformal homeomorphism. The "if"-direction is obvious.
    
    In the "only if"-direction, the fact that $\widetilde{\pi}$ is a homeomorphism follows from \Cref{thm:failure:terrible}. So $h = \psi \circ \widetilde{\pi} \colon \mathbb{S}^{2} \rightarrow \mathbb{S}^{2}$ is a homeomorphism satisfying $\Mod \Gamma \leq K \Mod h \Gamma$ for all path families $\Gamma$. \Cref{prop:williams:L-Wversion} and \cite[Definition 3.1.1 and Theorem 3.7.7]{Ast:Iwa:Mar:09} prove that $h$ is $K$-quasiconformal. Consequently, $\widetilde{\pi}$ is $K'$-quasiconformal for some $K' \leq K^{2}$. This self-improves to $K' = 1$ due to \Cref{lemm:regularity} below. This yields $\Mod \widetilde{\pi} \Gamma = \Mod \Gamma$ for all path families.
\end{proof}   

    \begin{lemm}\label{lemm:regularity}
    Suppose that $\widetilde{\pi} \colon \mathbb{S}^{2} \rightarrow \widetilde{Z}$ from \eqref{eq:quotientmap} is a homeomorphism. Then $\widetilde{\pi} \colon \mathbb{S}^{2} \rightarrow \widetilde{Z}$ is $1$-quasiconformal if and only if for every $1$-Lipschitz $h \colon \mathbb{S}^{2} \rightarrow \mathbb{R}$, $h \circ \widetilde{\pi}^{-1} \in N^{1,2}( \widetilde{Z} )$.
    \end{lemm}
    \begin{proof}
    The "only if"-claim is clear, given \Cref{prop:williams:L-Wversion} (ii). In the "if"-direction, fix a $1$-Lipschitz $h \colon \mathbb{S}^{2} \rightarrow \mathbb{R}$ for now.
    
    Consider the Borel function $G \colon \mathbb{S}^{2} \rightarrow \left[0, \infty\right]$ defined on \Cref{lemm:quotient:Sobolev}. Then $\rho = 1/G \circ \widetilde{\pi}^{-1}$ is such that $\rho^{2}$ is the Jacobian of $\widetilde{\pi}^{-1}$, as a consequence of \Cref{lemm:pullbackmeasure:modulus}. Hence $\rho \in L^{2}( \widetilde{Z} )$.
    
    Given that $h \circ \widetilde{\pi}^{-1} \in N^{1,2}( \widetilde{Z} )$ and $\mathcal{H}^{2}_{ \widetilde{Z} }( Q(S_{Z}) ) = 0$, for almost every $\gamma \colon \left[0, 1\right] \rightarrow \widetilde{Z}$, the composition $( h \circ \widetilde{\pi}^{-1} ) \circ \gamma$ is absolutely continuous, $\gamma$ has negligible length on the seam $Q( S_{Z} )$, and $\int_{ \gamma } \rho \,ds < \infty$. Indeed, the absolute continuity of $( h \circ \widetilde{\pi}^{-1} ) \circ \gamma$ for almost every path follows from \cite[Proposition 6.3.2]{HKST:15}. The fact that almost every path has negligible length on $Q( S_{Z} )$ follows from \Cref{lemm:negligible} and the $L^{2}$-integrability of $\infty \cdot \chi_{ Q(S_{Z}) }$. Similarly, the conclusion $\int_{ \gamma } \rho \,ds < \infty$ follows from \Cref{lemm:negligible} and the $L^{2}$-integrability of $\rho$.
    
    If we denote $E = ( h \circ \widetilde{\pi}^{-1} )( |\gamma| \cap Q(S_{Z}) )$, the absolute continuity of $( h \circ \widetilde{\pi}^{-1} ) \circ \gamma$ implies $\mathcal{H}^{1}_{\mathbb{R}}( E ) = 0$. Then \Cref{lemm:techincal:lemma} yields $\ell( ( h \circ \widetilde{\pi}^{-1} ) \circ \gamma ) \leq \int_{ \gamma } \rho \,ds$. We conclude that $\rho$ is a weak upper gradient of $h \circ \widetilde{\pi}^{-1}$.
    
    Since $\rho$ is independent of $h$ and $h$ is an arbitrary $1$-Lipschitz function, Theorem 7.1.20 \cite{HKST:15} shows that $\rho$ is a weak upper gradient of $\widetilde{\pi}^{-1}$. Since $\rho^{2}$ is the Jacobian of $\widetilde{\pi}^{-1}$, we conclude $K_{O}( \widetilde{\pi}^{-1} ) = 1$. Recall $K_{O}( \widetilde{\pi} ) = 1$ from \Cref{prop:quotientmap}.
    \end{proof}

    \begin{rem}
    If the welding curve $\mathcal{C}$ happens to be rectifiable, the Hausdorff $1$-measure on $\mathcal{C}$ and $\chi_{ \mathrm{Tn}( \mathcal{C} ) }\mathcal{H}^{1}_{ \mathcal{C} }$ are mutually absolutely continuous \cite[Chapter VI, Theorem 1.2 (F. and M. Riesz)]{Gar:Mar:05}. With this fact at hand, \Cref{lemm:mutualABS:tan} implies that $\widetilde{\pi}$ is a homeomorphism. Moreover, one can show that $h \circ \widetilde{\pi}^{-1} \in N^{1,2}( \widetilde{Z} )$ for every $1$-Lipschitz $h \colon \mathbb{S}^{2} \rightarrow \mathbb{R}$. Hence $\widetilde{\pi}$ is $1$-quasiconformal.
    \end{rem}
    
    \begin{proof}[Proof of \Cref{thm:welding:positive}]
    Suppose the existence of a quasiconformal homeomorphism $\psi \colon \widetilde{Z} \rightarrow \mathbb{S}^{2}$. Up to postcomposing $\psi$ by an orientation-reversing Möbius transformation of $\mathbb{S}^{2}$, we may assume that $\widetilde{\phi}_{i} \coloneqq \psi \circ \widetilde{\iota}_{i}|_{ Z_{i} } \colon Z_{i} \rightarrow \mathbb{S}^{2}$ is orientation-preserving for $i = 1,2$. Let $\mathcal{C} = \psi( Q( S_{Z} ) )$.
    
    The set $\mathbb{S}^{2} \setminus \mathcal{C}$ is the disjoint union of Jordan domains $\Omega_{1}$ and $\Omega_{2}$, where $\Omega_{i}$ is the image of $\widetilde{\phi}_{i}$ for $i = 1,2$. \Cref{thm:failure:terrible} implies $( Z, d_{Z} ) = \widetilde{Z}$.
    
    Next, since $\psi \colon \widetilde{Z} \rightarrow \mathbb{S}^{2}$ is a quasiconformal homeomorphism, $\psi$ satisfies Lusin's Condition ($N$) \cite[Section 17]{Raj:17}. Consequently, $\mathcal{C}$ has zero 2-dimensional Hausdorff measure.
    
    We consider the Beltrami differential $\mu = \chi_{ \Omega_{1} } \mu_{1} + \chi_{ \Omega_{2} } \mu_{2}$, where $\mu_{i}$ is the Beltrami differential of $\widetilde{\phi}_{i}^{-1}$. If $h \colon \mathbb{S}^{2} \rightarrow \mathbb{S}^{2}$ is a normalized solution to the Beltrami equation induced by $\mu$ \cite[Measurable Riemann mapping theorem]{Ast:Iwa:Mar:09}, the mapping $\widetilde{\psi} = h \circ \psi$ is $1$-quasiconformal. Since $\mathcal{C}$ has zero measure, this is readily verified by hand or by applying \cite[Theorem 4.12]{Iko:19}.

    We have verified that $( Z, d_{Z} ) = \widetilde{Z}$ and we may assume that $\psi \colon ( Z, d_{Z} ) \rightarrow \mathbb{S}^{2}$ is $1$-quasiconformal with $\phi_{i} = \psi \circ \widetilde{\iota}_{i}|_{ Z_{i} }$ being Riemann maps \cite[Weyl's lemma]{Ast:Iwa:Mar:09}. The definition of $Z$ implies that $g = \phi_{2}^{-1} \circ \phi_{1}|_{ \mathbb{S}^{1} }$. Consequently, $g$ is a welding homeomorphism.
    
    In order to show the removability of $\mathcal{C} \coloneqq \psi( S_{Z} )$, we are given an orientation-preserving homeomorphism $M \colon \mathbb{S}^{2} \rightarrow \mathbb{S}^{2}$ conformal in the complement of $\mathcal{C}$. Then $\pi' \coloneqq \psi^{-1} \circ M^{-1}$ defines a mapping as in \eqref{eq:quotientmap} for the curve $\mathcal{C}' = M( \mathcal{C} )$. \Cref{prop:quotientmap} implies that $\pi'$ is $1$-quasiconformal. Consequently, $M^{-1} = \psi \circ \pi'$ is $1$-quasiconformal, i.e., a Möbius transformation.
    \end{proof}

\section{Mass upper bound}\label{sec:Ahlfors}
    In this section, we prove Theorems \ref{thm:mass:accumulation} and \ref{cor:QC:Jacobianproblem}. We first consider the implication "(3) $\Rightarrow$ (1)". Recall that we are given an orientation-preserving homeomorphism $g \colon \mathbb{S}^{1} \rightarrow \mathbb{S}^{1}$ and the canonical quotient map $Q \colon Z \rightarrow \widetilde{Z}$. We are assuming the existence of a constant $C > 0$ for which
    \begin{equation}
        \label{eq:massupperbound:recall}
        \liminf_{ r \rightarrow 0^{+} }\frac{ \mathcal{H}^{2}_{ \widetilde{Z} }( \overline{B}_{\widetilde{Z}}( y, r ) ) }{ \pi r^2 } \leq C
        \quad\text{for every $y \in Q( S_{Z} )$}.
    \end{equation}
    
    In order to make transparent how $C$ and the Lipschitz constants of $g$ (resp. $g^{-1}$) are related to $C$ in \eqref{eq:massupperbound:recall}, we define $C_{1}, C_{2} \geq 0$ to be the smallest constants for which 
    \begin{align}
        \label{eq:massupperbound:recall:density:south}
        \liminf_{ r \rightarrow 0^{+} }
            \frac{ \mathcal{H}^{2}_{ \widetilde{Z} }( \overline{ \widetilde{\iota}_{1}( Z_{1} ) } \cap \overline{B}_{\widetilde{Z}}( y, r ) ) }{ \pi r^2 }
        &\leq
        C_{1}
        \quad\text{for every $y \in Q( S_{Z} )$}
        \\
        \label{eq:massupperbound:recall:density:north}
        \liminf_{ r \rightarrow 0^{+} }
            \frac{ \mathcal{H}^{2}_{ \widetilde{Z} }( \overline{ \widetilde{\iota}_{2}( Z_{2} ) } \cap \overline{B}_{\widetilde{Z}}( y, r ) ) }{ \pi r^2 }
        &\leq
        C_{2}
        \quad\text{for every $y \in Q( S_{Z} )$}.
    \end{align}
    Recalling from \Cref{lemm:inclusion} the fact that the inclusion maps are $1$-Lipschitz and local isometries outside the seam, the limit infimums in \eqref{eq:massupperbound:recall:density:south} and \eqref{eq:massupperbound:recall:density:north} are bounded from below by $1/2$. Hence, $C_{1}, C_{2} \geq 2^{-1}$. Since the seam is negligible, we have $1 \leq C_{1} + C_{2} \leq C$.
    
    We show that the constant $C_{1}$ in \eqref{eq:massupperbound:recall:density:south} and the Lipschitz constant $L_{1}$ of $g^{-1}$ are connected via the following function
    \begin{equation}
        \label{eq:lowerbound:inequality}
        f( \epsilon )
        \coloneqq
        \frac{ \left( \sin|_{(0,\pi/2]} \right)^{-1}( \epsilon ) }{ \pi }
        +
        \frac{ \sqrt{1-\epsilon^{2}} }{ \pi \epsilon }
        \quad
        \text{for $0 < \epsilon \leq 1$}.
    \end{equation}
    \begin{defi}\label{defi:uniquenumber}
    For every $C \geq 1/2$, $L = L( C ) \geq 1$ denotes the unique positive number such that for every $0 < \epsilon \leq L^{-1}$, $f( \epsilon ) \geq C$. Equivalently, $L = 1/f^{-1}(C)$.
    \end{defi}

    \begin{rem}
    We note that for every $0 < \epsilon \leq 1$, we have $f( \epsilon ) \geq ( \pi \epsilon )^{-1}$. We use this fact during the proof of \Cref{thm:mass:accumulation}.
    \end{rem}

    \begin{prop}\label{cor:Lipschitz}
    If \eqref{eq:massupperbound:recall:density:south} holds with constant $C_{1}$ and $L_1 = L( C_{1} )$ is as in \Cref{defi:uniquenumber}, then $g^{-1}$ is $L_1$-Lipschitz and $\widetilde{\iota}_{1} \colon \overline{Z}_{1} \rightarrow \widetilde{Z}$ satisfies for every $x, y \in \overline{Z}_{1}$, $\sigma( x, y ) \geq d_{Z}( \widetilde{\iota}_{1}(x), \widetilde{\iota}_{1}(y) ) \geq \sigma( x, y )/ L_{1}$.
    \end{prop}
    
    The symmetry in the argument yields the following result.
    \begin{prop}\label{cor:Lipschitz:north}
    If \eqref{eq:massupperbound:recall:density:north} holds with constant $C_{2}$ and $L_2 = L( C_{2} )$ is as in \Cref{defi:uniquenumber}, then $g$ is $L_2$-Lipschitz and $\widetilde{\iota}_{2} \colon Z_{2} \rightarrow \widetilde{Z}$ satisfies for every $x, y \in \overline{Z}_{2}$, $\sigma( x, y ) \geq d_{Z}( \widetilde{\iota}_{2}(x), \widetilde{\iota}_{2}(y) ) \geq \sigma( x, y )/L_{2}$.
    \end{prop}

    We start the proof of \Cref{cor:Lipschitz}. We consider the decomposition $g^{*}\mathcal{H}^{1}_{ \mathbb{S}^{1} } = v_{g}\mathcal{H}^{1}_{ \mathbb{S}^{1} } + \mu^{ \perp }$ with $\mu^{\perp}$ and $\mathcal{H}^{1}_{ \mathbb{S}^{1} }$ being singular. We fix a Borel representative of $v_{g}$. Let $f$ be as in \eqref{eq:lowerbound:inequality}. The following statement holds for every $\widetilde{Z}$.
    
    \begin{prop}\label{prop:massaccumulation}
    Given $1 > \epsilon > 0$ and a $\mathcal{H}^{1}_{ \mathbb{S}^{1} }$-density point $x_{0} \in \mathbb{S}^{1}$ of $E \coloneqq \left\{ v_{g} \leq \epsilon \right\}$, we have
    \begin{equation}
        \label{eq:masslowerbound:asymptotic}
        f( \epsilon )
        \leq
        \liminf_{ r \rightarrow 0^{+} }
            \frac{ \mathcal{H}^{2}_{ \widetilde{Z} }( \overline{ \widetilde{\iota}_{1}(Z_{1}) } \cap \overline{B}_{ \widetilde{Z} }( x_{0}, r ) ) }{ \pi r^{2} }.
    \end{equation}
    \end{prop}
    \begin{proof}
    For the duration of the proof, we fix normal coordinates $F \colon B( 0, \pi/2 ) \rightarrow \mathbb{S}^{2}$ centered at $x_{0}$ in such a way that the preimage of $\mathbb{S}^{1} \cap B( x_{0}, \pi/2 )$ is $\left( -\pi/2, \pi/2 \right) \times \left\{0\right\}$ \cite[Section 5]{Lee:18}. Recall that this means that $F$ is an isometry along radial geodesics and the metric has the expansion $g_{ij}(x) = \delta_{ij} + O( \norm{ x }_{2}^2 )$ in these coordinates. In particular, as $r \rightarrow 0^{+}$, the bi-Lipschitz constant of $F|_{ B( 0, r ) }$ is of the form $1 + O( r^2 )$. We denote $\Gamma(s) \coloneqq F(s,0)$ for $\abs{s} \leq \pi/2$.
    
    We fix $0 < \eta  < 1/\epsilon - 1$. Since $x_{0}$ is a density point of $E$, there exists $s_{0} < \pi / 2$ such that for every $0 < s \leq s_{0}$,
    \begin{equation}
        \label{eq:small:length}
        \mathcal{H}^{1}_{ \mathbb{S}^{1} }\left(
            \Gamma\left( \left[-s, s\right] \right) \setminus E
        \right)
        \leq
        \epsilon \eta s.
    \end{equation}
    We fix $0 < r \leq \epsilon s_{0}$. Then, for every $0 < s < r/\epsilon$, \Cref{lemm:hausdorff} yields for both $I = \left[0 , s\right]$ and $I = \left[-s, 0\right]$,
    \begin{equation}
        \label{eq:smallspeed}
        \ell( E \cap ( \widetilde{\iota}_{1} \circ \Gamma|_{I} ) )
        \leq
        \epsilon s.
    \end{equation}
    Since $\widetilde{\iota}_{1}$ is $1$-Lipschitz according to \Cref{lemm:inclusion}, \eqref{eq:small:length} and \eqref{eq:smallspeed} imply
    \begin{equation}
        \label{eq:length:computation}
        \ell( \widetilde{\iota}_{1} \circ \Gamma|_{ I } )
        \leq
        s \epsilon
        +
        \epsilon \eta s
        =
        \epsilon( 1 + \eta ) s
        <
        s.
    \end{equation}
    We denote for every $|s| < r/( ( 1 + \eta ) \epsilon )$, $\rho_{s} = r - \epsilon ( 1 + \eta ) |s|$. For each $z \in Z_{1} \cap B_{ \mathbb{S}^{2} }( F(s,0), \rho_s )$, the inequality \eqref{eq:length:computation} implies $\widetilde{\iota}_{1}(z) \in \overline{B}_{ \widetilde{Z} }( \widetilde{\iota}_{1}( x_{0} ), r )$.
    
    We estimate $A_{r} \coloneqq \mathcal{H}^{2}_{ \widetilde{Z} }( \overline{ \widetilde{\iota}_{1}(Z_{1}) } \cap \overline{B}_{ \widetilde{Z} }( \widetilde{\iota}_{1}( x_{0} ), r ) )$ as $r \rightarrow 0^{+}$. In estimating $A_{r}$, we use the fact that the seam $Q( S_{Z} )$ has negligible $\mathcal{H}^{2}_{ \widetilde{Z} }$-measure and that $\widetilde{\iota}_{1}$ is a local isometry outside the seam. We claim that for each $0 < \theta < \pi/2$ the following holds:
    \begin{equation}
        \label{eq:lowerbound:proof:initial}
        A_{r}
        \geq
        ( 1 + O( (r/\epsilon)^2 ) )^{-2}
        \left(
            \theta r^2
            +
            \cos( \theta )
            \frac{ r^{2} }{ ( 1 + O( (r/\epsilon)^2 ) ) \epsilon ( 1+ \eta) }
        \right).
    \end{equation}
    The term $( 1 + O( (r/\epsilon)^2 ) )^{-2}$ comes from estimating the Jacobian of $\widetilde{\iota}_{1} \circ F$. The first term in the brackets comes from the fact that $F$ preserves the speed of radial geodesics, so
    \begin{equation*}
        ( \widetilde{\iota}_{1} \circ F )\left( \left\{ (s,t) \colon \sqrt{ s^{2} + t^{2} } < r, 0 < t \right\} \right)
        \subset
        \overline{ \widetilde{\iota}_{1}( Z_{1} ) }
        \cap
        \overline{B}_{ \widetilde{Z} }( \widetilde{\iota}_{1}( x_{0} ), r ).
    \end{equation*}
    We use this inclusion in a circular sector $C_{\theta}(r)$ which has a total angle $2\theta$ and an angle bisector $\left\{0\right\} \times \mathbb{R}$.
    
    The second term in the brackets is twice the area of a suitable triangle. The factor of two comes from the symmetry of the estimate \eqref{eq:length:computation} with respect to the parameter $s = 0$. We consider a triagle $T_{\theta}(r) \subset \mathbb{R}^{2}$ foliated by line segments $\ell(s)$, where $0 \leq s < r/ ( ( 1 + \eta ) \epsilon )$, with $\ell(s)$ having the start point $(s,0)$, tangent in the direction $( \sin(\theta), \cos(\theta) )$, and has length $\rho_{a}/( 1 + O( (r/\epsilon)^2 ) )$. The $\widetilde{\iota}_{1} \circ F$ image of such a triangle $T_{\theta}(r)$ contributes to $A_{r}$. The inequality \eqref{eq:lowerbound:proof:initial} follows.
    
    We choose the angle $\theta$ to satisfy $\sin( \theta ) = \epsilon( 1 + \eta )$. We divide \eqref{eq:lowerbound:proof:initial} by $\pi r^2$, pass to the limit $r \rightarrow 0^{+}$, and then $\eta \rightarrow 0^{+}$ and conclude
    \begin{equation}
        \label{eq:lowerbound:proof}
        \liminf_{ r \rightarrow 0^{+} }
        \frac{ A_{r} }{ \pi r^2 }
        \geq
        \frac{ \left( \sin|_{ (0,\pi/2] } \right)^{-1}( \epsilon ) }{ \pi }
        +
        \frac{ \sqrt{1-\epsilon^{2}} }{ \pi \epsilon }
        =
        f(\epsilon).
    \end{equation}
    The inequality \eqref{eq:masslowerbound:asymptotic} is the same as \eqref{eq:lowerbound:proof}.
    \end{proof}
    
    \begin{rem}
    Given $0 < \epsilon < 1$, the lower bound in \eqref{eq:lowerbound:proof} is sharp. This can be shown by considering a bi-Lipschitz $g \colon \mathbb{S}^{1} \rightarrow \mathbb{S}^{1}$ with metric speed $v_{g} \equiv \epsilon$ everywhere in an open neighbourhood of $x_{0} \in \mathbb{S}^{1}$.
    
    If the circular sector $C_{\theta}(r)$ and triangle $T_{\theta}(r)$ are defined as in the proof of the lower bound \eqref{eq:lowerbound:proof}, with $\eta = 0$, and $\theta = \left( \sin|_{ (0,\pi/2] } \right)^{-1}( \epsilon )$, we have
    \begin{equation*}
        \liminf_{ r \rightarrow 0^{+} }
        \frac{ A_{r} }{ \pi r^{2} }
        =
        \frac{ \mathcal{H}^{2}_{ \mathbb{R}^{2} }( C_{\theta}(1) ) + \mathcal{H}^{2}_{ \mathbb{R}^{2} }( T_{\theta}(1) ) }{ \pi }
        =
        f( \epsilon ).
    \end{equation*}
    This can be showed using \Cref{lemm:jumping:overtheseam} and \Cref{lemm:hausdorff}. The key property of the angle $\theta$ is that the line on $\mathbb{R}^{2}$ containing $( r/ \epsilon, 0 )$ with tangent vector $( -\cos(\theta), \sin( \theta ) )$ intersects every ball $\overline{B}_{ \mathbb{R}^{2} }( (s,0), \rho_{s} )$ tangentially when $0 \leq s < 1/\epsilon$ and $\rho_{s} = 1 - \epsilon s$.
    \end{rem}
    
    \begin{proof}[Proof of \Cref{cor:Lipschitz}]
    Given \eqref{eq:massupperbound:recall:density:south} and \Cref{prop:massaccumulation}, we have $v_{g}(x) \geq L_{1}^{-1}$ for $\mathcal{H}^{1}_{ \mathbb{S}^{1} }$-almost every $x \in \mathbb{S}^{1}$. This implies that $g^{-1}$ is absolutely continuous and $v_{ g^{-1} }( x ) \leq L_1$ for $\mathcal{H}^{1}_{ \mathbb{S}^{1} }$-almost every $x \in \mathbb{S}^{1}$. Therefore $g^{-1}$ is $L_1$-Lipschitz.
    
    The fact that $\widetilde{\iota}_{1}$ $1$-Lipschitz follows from \Cref{lemm:inclusion}. \Cref{lemm:hausdorff} implies that
    \begin{equation*}
        d_{Z}( \widetilde{\iota}_{1}(x), \widetilde{\iota}_{1}(y) ) \geq \sigma( x, y )/L_{1}
        \quad\text{for every $x, y \in \mathbb{S}^{1}$}.
    \end{equation*}
    The equality \Cref{lemm:jumping:overtheseam} \eqref{eq:energyminimal:chain} implies the corresponding inequality for every pair $x, y \in \overline{Z}_{1}$. Hence $\widetilde{\iota}_{1}^{-1}$ is $L_{1}$-Lipschitz.
    \end{proof}

    Next, we verify a lemma about radial extensions of bi-Lipschitz maps, which we need during the proof of \Cref{thm:mass:accumulation}.
    
    For the south pole $P_{1} \in Z_{1}$, we consider the (orientation-reversing) stereographic projection $P \colon \mathbb{S}^{2} \setminus \left\{ P_{1} \right\} \rightarrow \mathbb{R}^{2} \times \left\{0\right\}$ fixing the equator and mapping the north pole $P_{2} = (0,0,1)$ to the origin. We identify $\mathbb{R}^{2} \times \left\{0\right\}$ with $\mathbb{R}^{2}$. We note that $P^{-1}$ has the explicit definition
    \begin{equation*}
        P^{-1}(x,y)
        =
        \left(
            \frac{ 2 x }{ 1 + x^2 + y^2 },
            \frac{ 2 y }{ 1 + x^2 + y^2 },
            \frac{ 1 - x^2 + y^2 }{ 1 + x^2 + y^2 }
        \right).
    \end{equation*}
    The Riemannian tensor of $\mathbb{S}^{2}$ in these coordinates is $I = ( 4 / ( 1 + r^2 )^2 )g_{E}$, where $r$ is the distance to the origin and $g_{E}$ the Euclidean inner product. In polar coordinates, $g_{E} = dr^2 + r^2 d\theta^2$. 
    We see from the form of $I$ that the bi-Lipschitz constant of $\widetilde{g} = P \circ g \circ ( P|_{ \mathbb{S}^{1} } )^{-1}$ and $g \colon \mathbb{S}^{1} \rightarrow \mathbb{S}^{1}$ coincide.
    
    We represent the polar coordinates using the complex notation $r e^{ i \theta }$. We note that there exists a homeomorphism $\widetilde{G} \colon \mathbb{R} \rightarrow \mathbb{R}$ with $\widetilde{g}( e^{ i \theta } ) = e^{ i \widetilde{G}( \theta ) }$ for every $\theta \in \mathbb{R}$. For every $0 \leq r \leq 1$ and $\theta \in \mathbb{R}$, we set $\widetilde{\psi}( r e^{ i \theta } ) \coloneqq r e^{ i \widetilde{G}( \theta ) }$ and refer to $\widetilde{\psi}$ as the \emph{radial extension} of $\widetilde{g}$. We recall from \cite[Theorem 2.2]{Kal:14} that the bi-Lipschitz constants of $\widetilde{g}$ and $\widetilde{\psi}$ coincide. Let $\psi = P^{-1} \circ \widetilde{\psi} \circ P|_{Z_{2}} \colon Z_{2} \rightarrow Z_{2}$.
    
    We use the following fact during the proof of \Cref{lemm:radial}; see for example \cite{DC:Jar:Sha:16}, \cite{Cre:Sou:20}.
    \begin{lemm}\label{lemm:pencil}
    For every $x, y \in \mathbb{S}^{2}$, $0 < \epsilon < 1$, and $0 < 4 r < \sigma( x, y )$, the modulus of the family of paths joining $B_{ \mathbb{S}^{2} }( x, r )$ to $B_{ \mathbb{S}^{2} }( y, r )$ with length $( 1 + \epsilon ) \sigma( x, y )$ is positive.
    \end{lemm}
    
    \begin{lemm}\label{lemm:radial}
    The map $\psi \colon Z_{2} \rightarrow Z_{2}$ is $L$-bi-Lipschitz if $g$ is $L$-bi-Lipschitz.
    \end{lemm}
    \begin{proof}
    We refer the interested reader to \cite[Section 2]{Kal:14} for the proof of the fact that $\widetilde{\psi}$ is bi-Lipschitz if $\widetilde{g}$ (equivalently $g$) is bi-Lipschitz. We take this as a given.
    
    Since $\widetilde{\psi}$ is bi-Lipschitz, it has a differential at $\mathcal{L}^{2}$-almost every point in $\mathbb{D}$. Given this fact, the following computations are understood to hold at $\mathcal{L}^{2}$-almost every $( x, y ) = r e^{ i \theta }$ in the unit disk.
    
    The pullback $\widetilde{\psi}^{*}I$ is a diagonal matrix with respect to the basis $(dr, d\theta)$, with diagonal $4/(1+r^2)^2$ and $4 |\widetilde{G}'( \theta ) |^2 r^2 / ( 1 + r^2 )^2$. Hence the maximum of the operator norms of $D\widetilde{\psi} \colon ( T\mathbb{D}, I ) \rightarrow ( T\mathbb{D}, I )$ and its inverse is equal to $L( r e^{i\theta} ) = \max\left\{ |\widetilde{G}'( \theta )|, |\widetilde{G}'( \theta )|^{-1} \right\}$. Then, if $L'$ denotes the essential supremum of $L(r e^{ i \theta } )$, \Cref{lemm:pencil} implies that $\psi$ is $L'$-bi-Lipschitz. On the other hand, $L'$ is the bi-Lipschitz constant of $g$.
    \end{proof}
    
    \begin{proof}[Proof of \Cref{thm:mass:accumulation}]
    We first claim that "(1) $\Rightarrow$ (2)". \Cref{lemm:radial} provides us with an $L$-bi-Lipschitz $\psi \colon Z_{2} \rightarrow Z_{2}$ extension of the given $L$-bi-Lipschitz $g$. We define $H(x) = \widetilde{\iota}_{1}(x)$ for each $x \in \overline{Z}_{1}$ and $H(x) = \widetilde{\iota}_{2} \circ \psi(x)$ otherwise. \Cref{lemm:hausdorff} implies that $H$ is $L$-bi-Lipschitz at the seam, and \Cref{lemm:jumping:overtheseam} implies that $H$ is $L$-bi-Lipschitz everywhere.
    
    Notice that if $H \colon \mathbb{S}^{2} \rightarrow \widetilde{Z}$ is $L'$-bi-Lipschitz, we may choose $C = (L')^{4}$ as an upper bound for the 2-dimensional Hausdorff lower density. Hence "(2) $\Rightarrow$ (3)" follows, quantitatively. Lastly, "(3) $\Rightarrow$ (1)" follows from Propositions \ref{cor:Lipschitz} and \ref{cor:Lipschitz:north}. In fact, given $C \geq 1$ for which the lower density bound \eqref{eq:massupperbound:recall} holds, $g$ is $L'$-bi-Lipschitz for $L'$ solving $C = f(1/L')$. Since $f(\epsilon) \geq 1/\pi \epsilon$ for every $0 < \epsilon \leq 1$, we have $C \pi \geq L'$. Hence $g$ is $C \pi$-bi-Lipschitz.
    \end{proof}
    
    \begin{rem}\label{rem:sharpness}
    The estimates between the constants in "$(3) \Rightarrow (1)$" in \Cref{thm:mass:accumulation} can be improved in two ways. First, the constants $C_{1}$ and $C_{2}$ in \eqref{eq:massupperbound:recall:density:south} and \eqref{eq:massupperbound:recall:density:north} satisfy $\max\left\{ C_{1}, C_{2} \right\} \leq C - 1/2$, so $g$ is $ ( C - 1/2 )\pi$-bi-Lipschitz.
    
    The second improvement is obtained by using the constant $L' = L( C - 1/2 )$ from \Cref{defi:uniquenumber}. Then $g$ is $L'$-bi-Lipschitz, where $L' \leq ( C - 1/2) \pi$.
    
    These improvements imply that the bi-Lipschitz constant of $g$ converges to $1$ as $C \rightarrow 1^{+}$. These facts also improve \Cref{cor:QC:Jacobianproblem} and the following result, \Cref{thm:QC:Jacobianproblem}.
    \end{rem}
    
    Before proving \Cref{cor:QC:Jacobianproblem}, we investigate a related problem. To this end, suppose that we are given Riemann maps $\phi_{i} \colon Z_{i} \rightarrow \Omega_{i}$ with $\Omega_{1}$ and $\Omega_{2}$ denoting the complementary components of a welding curve $\mathcal{C}$, and set $g = \phi_{2}^{-1} \circ \phi_{1}|_{ \mathbb{S}^{1} }$.
\begin{prop}\label{thm:QC:Jacobianproblem}
    Let $K, C \geq 1$. The welding homeomorphism $g$ is $\pi (KC)^{2}$-bi-Lipschitz if there exists a $K$-quasiconformal homeomorphism $h \colon \mathbb{S}^{2} \rightarrow \mathbb{S}^{2}$ such that for both $i = 1,2$,
    \begin{equation}
        \label{eq:jacobian:comparable}
        C^{-1}
        J_{ h }(x)
        \leq
        J_{ \phi_{i}^{-1} }(x)
        \leq
        CJ_{h}(x)
        \quad\text{ for $\mathcal{H}^{2}_{ \mathbb{S}^{2} }$-a.e. $x \in \Omega_{i}$.}
    \end{equation}
    Conversely, if $g$ is $L$-bi-Lipschitz, then there exists $L^{4}$-quasiconformal homeomorphism $h \colon \mathbb{S}^{2} \rightarrow \mathbb{S}^{2}$ such that \eqref{eq:jacobian:comparable} holds for $C = L^{2}$.
\end{prop}

    \begin{proof}
    We first assume that $g \colon \mathbb{S}^{1} \rightarrow \mathbb{S}^{1}$ is $L$-bi-Lipschitz. Then \Cref{thm:mass:accumulation} provides us with an $L$-bi-Lipschitz homeomorphism $\Psi \colon \widetilde{Z} \rightarrow \mathbb{S}^{2}$. \Cref{prop:quotientmap} and \eqref{eq:quotientmap} imply that $\widetilde{\pi} \colon \mathbb{S}^{2} \rightarrow \widetilde{Z}$ defined via the formula
    \begin{equation}
        \label{eq:1-QC:recall}
        \widetilde{ \pi }(x)
        =
        \left\{
        \begin{split}
            &\widetilde{\iota}_{1} \circ \phi_{1}^{-1}(x),
            &&x \in \overline{ \Omega_{1} },
            \\
            &\widetilde{\iota}_{2} \circ \phi_{2}^{-1}(x),
            &&x \in \Omega_{2}
        \end{split}
        \right.
    \end{equation}
    is a $1$-quasiconformal homeomorphism. Therefore, $h \coloneqq \Psi \circ \widetilde{\pi} \colon \mathbb{S}^{2} \rightarrow \mathbb{S}^{2}$ is $K$-quasiconformal for $K = L^4$, and as $\Psi$ is $L$-bi-Lipschitz, the Jacobians of $h$ and $\widetilde{\pi}$ are comparable with comparison constant $C = L^2$.
    
    Next, we are given a Jordan curve $\mathcal{C} \subset \mathbb{S}^{2}$ corresponding to a welding homeomorphism $g = \phi_{2}^{-1} \circ \phi_{1}|_{ \mathbb{S}^{1} }$, a $K$-quasiconformal homeomorphism $h \colon \mathbb{S}^{2} \rightarrow \mathbb{S}^{2}$, and a constant $C \geq 1$ such that
    \begin{equation}
        \label{eq:comparison}
        C^{-1}
        J_{ h }(x)
        \leq
        J_{ \widetilde{\pi} }(x)
        \leq
        C
        J_{ h }(x)
        \quad\mathcal{H}^{2}_{ \mathbb{S}^{2} }\text{-a.e. $x \in \mathbb{S}^{2} \setminus \mathcal{C}$}.
    \end{equation}
    For $i = 1,2$, the composition $h \circ \phi_{i}$ is $K$-quasiconformal with Jacobian bounded from above $C$ and below by $C^{-1}$, respectively; here we apply \eqref{eq:comparison}. \Cref{prop:williams:L-Wversion} (ii) and Hadamard's inequality imply that $C^{-1} \leq \rho_{ h \circ \phi_{i} }^{2} \leq K C$ $\mathcal{H}^{2}_{ \mathbb{S}^{2} }$-almost everywhere in $Z_{i}$. \Cref{lemm:pencil} implies that the homeomorphism $h \circ \phi_{i}$ is locally $L'$-bi-Lipschitz for $L' = \sqrt{KC}$.
    
    Since, for both $i = 1, 2$, $\overline{Z}_{i}$ is geodesic, it is immediate that $h \circ \phi_{i} \colon \overline{Z}_{i} \rightarrow \mathbb{S}^{2}$ is $L'$-Lipschitz. Since this holds for both $i = 1,2$, the construction of $d_{Z}$ implies that whenever $x, y \in \mathbb{S}^{1}$, $\sigma( h \circ \phi_{1}(x), h \circ \phi_{1}(y) ) \leq L' d_{\widetilde{Z}}( \widetilde{\iota}_{1}(x), \widetilde{\iota}_{1}(y) )$. \Cref{lemm:jumping:overtheseam} \eqref{eq:energyminimal:chain} establishes the same inequality for each $x, y \in \overline{Z}_{1}$. Hence the mapping $\widetilde{\pi}$ defined by the expression \eqref{eq:1-QC:recall} is a homeomorphism and $\Psi \coloneqq h \circ \widetilde{\pi}^{-1}$ is $L'$-Lipschitz on the southern hemisphere. A similar argument shows that $\Psi$ is $L'$-Lipschitz on both of the hemispheres. Then \Cref{lemm:jumping:overtheseam} \eqref{eq:energyminimal:chain:seam} implies that $\Psi$ is $L'$-Lipschitz everywhere.
    
    Since $\Mod \Gamma \leq K \Mod \Psi^{-1} \Gamma$ for all path families (recall \Cref{prop:quotientmap}), we have $\Psi^{-1} \in N^{1,2}( \mathbb{S}^{2}, \widetilde{Z} )$. On the other hand, $\Psi( Q(S_{Z}) )$ has negligible $\mathcal{H}^{2}_{ \mathbb{S}^{2} }$-measure and $\Psi^{-1}$ is locally $L'$-Lipschitz in the complement of that set. In particular, almost every absolutely continuous $\gamma \colon \left[0, 1\right] \rightarrow \mathbb{S}^{2}$ has zero length in $\Psi( Q( S_{Z} ) )$ and $\Psi^{-1} \circ \gamma$ is absolutely continuous. As a consequence, $\mathcal{H}^{1}_{ \widetilde{Z} }( Q( S_{Z} ) \cap | \Psi^{-1} \circ \gamma | ) = 0$.
    
    Denoting $E = Q( S_{Z} ) \cap | \Psi^{-1} \circ \gamma |$ and $\rho = L' \chi_{ \mathbb{S}^{2} }$, we conclude from \Cref{lemm:techincal:lemma} that $\ell( \Psi^{-1} \circ \gamma ) \leq \int_{ \gamma } \rho \,ds \leq L' \ell( \gamma )$. \Cref{lemm:pencil} implies that $\Psi^{-1}$ is $L'$-Lipschitz.
    
    We have verified that $\Psi$ is $L'$-bi-Lipschitz. By applying the implications "(2) $\Rightarrow$ (3) $\Rightarrow$ (1)" in \Cref{thm:mass:accumulation}, we conclude that $g$ is $L$-bi-Lipschitz for $L = \pi (L')^{4} = \pi (KC)^{2}$.
    \end{proof}
    Next, we prove \Cref{cor:QC:Jacobianproblem}. This essentially follows from \Cref{thm:QC:Jacobianproblem}.
    \begin{proof}[Proof of \Cref{cor:QC:Jacobianproblem}]
    We claim that $g \colon \mathbb{S}^{1} \rightarrow \mathbb{S}^{1}$ is bi-Lipschitz if and only if there exists a quasiconformal homeomorphism $h \colon \mathbb{S}^{2} \rightarrow \mathbb{S}^{2}$ and a $1$-quasiconformal homeomorphism $\varphi \colon \mathbb{S}^{2} \rightarrow \widetilde{Z}$ such that $J_{\varphi}$ and $J_{ h }$ are comparable.
    
    If such $\varphi$ and $h$ exist, we may assume that $\phi_{i} = \varphi^{-1} \circ \widetilde{\iota}_{i}|_{ Z_{i} }$ is a Riemann map for both $i = 1,2$. Then \Cref{thm:QC:Jacobianproblem} shows that $g$ is bi-Lipschitz.
    
    Conversely, if $g$ is bi-Lipschitz, \Cref{thm:mass:accumulation} provides a bi-Lipschitz homeomorphism $\Psi \colon \widetilde{Z} \rightarrow \mathbb{S}^{2}$. Then \Cref{thm:welding:positive} implies the existence of a $1$-quasiconformal homeomorphism $\pi \colon \mathbb{S}^{2} \rightarrow \widetilde{Z}$ such that $\phi_{i} = \pi^{-1} \circ \widetilde{\iota}_{i}|_{ Z_{i} }$ is a Riemann map for $i = 1,2$. We may also assume that $\Psi \circ \widetilde{\iota}_{i}|_{ Z_{i} }$ is orientation-preserving for $i = 1,2$, by post-composing $\Psi$ with a suitable reflection, if need be. Defining $h = \Psi \circ \pi$ implies that the assumptions of \Cref{thm:QC:Jacobianproblem} hold for $g$.
    
    Since \Cref{thm:mass:accumulation} and \Cref{thm:QC:Jacobianproblem} are quantitative, so is \Cref{cor:QC:Jacobianproblem}.
    \end{proof}

\section{Mappings of finite distortion}\label{sec:finite:distortion}
    In this section, we establish \Cref{thm:QS:abs} and \Cref{thm:biconformal:abs}.

    \begin{defi}
    Let $\Omega, \Omega' \subset \mathbb{S}^{2}$ be open. A homeomorphism $\psi \colon \Omega \rightarrow \Omega'$ is a \emph{mapping of finite distortion} if $\psi \in N^{1,1}( \Omega, \mathbb{S}^{2} )$; second, the determinant $J( D\psi )$ of the differential $D\psi$ is nonnegative and integrable; lastly, there exists a function $1 \leq K_{\psi}' < \infty$ for which
    \begin{equation}
        \label{eq:definition:pointwisedistortion}
        \abs{ D\psi }_{g}^{2}
        \leq
        K_{ \psi }' J( D\psi )
        \quad
        \mathcal{H}^{2}_{ \mathbb{S}^{2} }\text{-a.e. in $\Omega$.}
    \end{equation}
    Here $\abs{ D\psi }_{g}$ refers to the operator norm of the differential $D\psi$. We let $K_{\psi}$ denote a smallest Borel function which is bounded from below by $\chi_{ \Omega  }$ and for which \eqref{eq:definition:pointwisedistortion} holds.
    \end{defi}

    \begin{defi}\label{defi:admissibledistortion}
    A smooth strictly increasing function $\mathcal{A} \colon \left[1, \infty\right) \rightarrow \left[0, \infty\right)$ is \emph{admissible} if
    \begin{enumerate}
        \item $\mathcal{A}(1) = 0$,
        \item $\int_{ 1 }^{ \infty } t^{-2}\mathcal{A}(t) \,d\mathcal{L}^{1}(t) = \infty$, and
        \item $t \mapsto t \mathcal{A}'(t)$ is increasing for large values $t$, and converges to $\infty$ as $t \rightarrow \infty$.
    \end{enumerate}
    \end{defi}

    We obtain the same class of admissible $\mathcal{A}$ if we replace (2) with the condition
    \begin{equation*}
        \int_{ 1 }^{ \infty } t^{-1} \mathcal{A}'(t) \,d\mathcal{L}^{1}(t) = \infty.
    \end{equation*}
    This follows from the fact that $\mathcal{A}(s)/s \leq 4 \int_{ s }^{ 2 s } t^{-2} \mathcal{A}(t) \,d\mathcal{L}^{1}(t)$ whenever $s \geq 1$ and the integration by parts formula.

    \begin{defi}\label{defi:admissible}
    Let $\Omega, \Omega' \subset \mathbb{S}^{2}$ be open, and $\psi \colon \Omega \rightarrow \Omega'$ a homeomorphism. We say that $\psi$ is \emph{admissible} if $\psi$ is a mapping of finite distortion and there exists an admissible $\mathcal{A}$ with
    \begin{equation}
        \label{eq:exponentialintegrability}
        \int_{ \Omega }
            e^{ \mathcal{A}( K_{\psi} ) }
        \,d\mathcal{H}^{2}_{ \mathbb{S}^{2} }
        <
        \infty.
    \end{equation}
    If $\mathcal{A}(t) = pt - p$ for some $p > 0$, we say that $\psi$ has \emph{exponentially integrable distortion}.
    \end{defi}
    
    We recall some properties of such $\psi$. First, $\psi$ satisfies Lusin's Condition ($N$) \cite[Theorem 1.1]{Ka:Ko:Ma:On:Zh:03}. Second, $\psi^{-1} \in N^{1,2}( \Omega', \Omega )$ \cite[Corollary 1.2]{Ko:On:06}; this implies that $\psi^{-1}$ satisfies Lusin's Condition ($N$) \cite[Theorem 3.3.7]{Ast:Iwa:Mar:09}. Third, the Jacobian $J( D\psi )$ appearing on the right-hand side of \eqref{eq:definition:pointwisedistortion} coincides with the Jacobian $J_{\psi}$ we defined in \Cref{sec:sobolev} \cite{Ka:Ko:Ma:On:Zh:03}.
    
    In this section, we show the following theorem.
    \begin{thm}\label{thm:mod}
    Suppose that $g \colon \mathbb{S}^{1} \rightarrow \mathbb{S}^{1}$ is a homeomorphism, $g^{-1}$ absolutely continuous, and there exists a homeomorphism $\psi \colon \overline{Z}_{2} \rightarrow \overline{Z}_{2}$ extending $g$ with $\psi|_{ Z_{2} }$ admissible. Then $\widetilde{Z}$ is quasiconformally equivalent to $\mathbb{S}^{2}$.
    \end{thm}
    Note that \Cref{thm:biconformal:abs} is a consequence of \Cref{thm:mod} so it suffices to verify \Cref{thm:mod}.
\begin{defi}\label{defi:standingassumptions}
    Given $x_{0} \in \mathbb{S}^{1}$ and $\pi > R_{0} > 0$, set $\widetilde{Q} \coloneqq \overline{B}_{ \mathbb{S}^{2} }( x_{0}, R_{0} ) \subset \mathbb{S}^{2}$. We define $H(x) = \widetilde{\iota}_{1}(x)$ if $x \in \widetilde{Q} \cap \overline{Z}_{1}$ and $\widetilde{\iota}_{2} \circ \psi(x)$ if $x \in \widetilde{Q} \cap Z_{2}$, and denote $\widetilde{R} = H( \widetilde{Q} ) \subset \widetilde{Z}$.
\end{defi}

\begin{prop}\label{thm:local:recip}
    If $\widetilde{R}$ and $H$ are as in \Cref{defi:standingassumptions}, then $H$ is a homeomorphism and there exists a $1$-quasiconformal homeomorphism $f = (u, v) \colon \widetilde{R} \rightarrow \left[0, 1\right] \times \left[0, M\right]$ for some $M > 0$.
\end{prop}

\begin{proof}[Proof of \Cref{thm:mod} assuming \Cref{thm:local:recip}]
    We cover the seam in $\widetilde{Z}$ by the interiors of $\widetilde{R}$ as in \Cref{defi:standingassumptions}. This implies that $\widetilde{Z}$ can be covered by quasiconformal images of planar domains, and the quasiconformal equivalence of $\widetilde{Z}$ and $\mathbb{S}^{2}$ follows from \Cref{thm:iko}.
\end{proof}

    The following lemma is a key step in proving \Cref{thm:local:recip}.
\begin{lemm}\label{lemm:measurable:sobolev}
       The $H$ from \Cref{defi:standingassumptions} is a homeomorphism, $H \in N^{1,1}( \widetilde{Q}, \widetilde{R} )$ and $H^{-1} \in N^{1,2}( \widetilde{R}, \widetilde{Q} )$. Furthermore, $H$ satisfies Lusin's Conditions ($N$) and ($N^{-1}$).
\end{lemm}
\begin{proof}
    The absolute continuity of $g^{-1}$ implies for the Lebesgue decomposition $g^{*}\mathcal{H}^{1} = v_{g} \mathcal{H}^{1} + \mu^{\perp}$ that $\left\{ v_{g} = 0 \right\}$ has negligible $\mathcal{H}^{1}_{ \mathbb{S}^{1} }$-measure in an open neighbourhood of $\mathbb{S}^{1} \cap \widetilde{Q}$. Then \Cref{lemm:hausdorff} and \Cref{lemm:jumping:overtheseam} imply that $H$ is a homeomorphism.
    
    We recall from \Cref{lemm:inclusion} the fact that the inclusion maps $\widetilde{\iota}_{1}|_{ Z_{1} } \colon Z_{1} \rightarrow \widetilde{Z}$ and $\widetilde{\iota}_{2}|_{ Z_{2} } \colon Z_{2} \rightarrow \widetilde{Z}$ are $1$-Lipschitz local isometries. This implies that $H$ and its inverse are absolutely continuous in measure; the seam has negligible Hausdorff $2$-measure.
    
    In the following proof, we write $\widetilde{\rho}_{i}$ for functions defined on $\widetilde{Q} \cap Z_{i} \subset \mathbb{S}^{2}$ and $\rho_{i} = ( \widetilde{\rho}_{i} \circ \widetilde{\iota}_{i}^{-1} )$ on $\widetilde{R} \cap \widetilde{\iota}_{i}( Z_{i} ) \subset \widetilde{Z}$ for $i = 1,2$.
    
    Since $\psi^{-1} \in N^{1,2}( \widetilde{Q} \cap Z_{2}, \mathbb{S}^{2} )$, for $i = 1, 2$, there exists an upper gradient $\widetilde{\rho}_{i} \in L^{2}( \widetilde{Q} \cap Z_{i} )$ of $H^{-1} \circ \widetilde{\iota}_{i}|_{ Z_{i} \cap \widetilde{Q} }$ for $i = 1,2$. We fix such functions and denote $\rho \coloneqq \chi_{ \widetilde{R} \cap \widetilde{\iota}_{1}( Z_{1} ) }\rho_{1} + \chi_{ \widetilde{R} \cap \iota_{2}( Z_{2} ) }\rho_{2} \in L^{2}( \widetilde{R} )$.

    Let $\Gamma_{0}$ denote the collection of non-constant paths on $\widetilde{R} \subset \widetilde{Z}$ which have positive length in the seam $Q( S_{Z} )$ or along which $\rho$ fails to be integrable. Since $\rho + \infty \cdot \chi_{ Q( S_{Z} ) }$ is $L^{2}$-integrable, \Cref{lemm:negligible} yields $\Mod \Gamma_{0} = 0$.

    Consider next an absolutely continuous path $\gamma \colon \left[0, 1\right] \rightarrow \widetilde{R}$ in the complement of $\Gamma_{0}$. Then $\theta = H^{-1} \circ \gamma$ is such that $\mathcal{H}^{1}_{ \mathbb{S}^{2} }( | \theta | \cap \mathbb{S}^{1} ) = 0$. Indeed, since $\gamma$ has zero length in the seam, the area formula \eqref{eq:areaformula} implies $\mathcal{H}^{1}_{ \widetilde{Z} }( | \gamma | \cap Q( S_Z ) ) = 0$. This implies $\mathcal{H}^{1}_{ \mathbb{S}^{1} }( | \theta | \cap \mathbb{S}^{1} ) = 0$ due to \Cref{lemm:hausdorff} and the absolute continuity of $g|_{ \mathbb{S}^{1} \cap \widetilde{Q} }^{-1}$. Since $\mathcal{H}^{1}_{ \mathbb{S}^{1} }( | \theta | \cap \mathbb{S}^{1} ) = 0$, the assumptions of \Cref{lemm:techincal:lemma} are satisfied. Hence
    \begin{equation*}
        \ell( \theta )
        \leq
        \int_{ \gamma }
            \rho
        \,ds
        <
        \infty.
    \end{equation*}
    This implies that $H^{-1}$ has an $L^{2}$-integrable weak gradient, so $H^{-1} \in N^{1,2}( \widetilde{R}, \widetilde{Q} )$.

    Lastly, we claim that $H \in N^{1,1}( \widetilde{Q}, \widetilde{R} )$. To this end, we observe that $H_{ \widetilde{Q} \cap Z_{i} }$ has an upper gradient $\widetilde{ \rho }_{i} \in L^{1}( \widetilde{Q} \cap Z_{i} )$, and denote $\widetilde{ \rho } = \sum_{ i = 1 }^{ 2 } \chi_{ \widetilde{Q} \cap Z_{i} } \widetilde{ \rho }_{i} \in L^{1}( \widetilde{Q} )$. Now $\widetilde{ \rho }$ is integrable along $1$-almost every absolutely continuous path $\gamma \colon \left[0, 1\right] \rightarrow \widetilde{Q}$ and $1$-almost every such path has zero length in $\mathbb{S}^{1}$. Having fixed a path $\gamma$ with these properties, \Cref{lemm:hausdorff} implies that $\theta = H \circ \gamma$ has zero length in the seam. The inequality $\ell( \theta ) \leq \int_{ \gamma } \rho \,ds$ follows from \Cref{lemm:techincal:lemma}. This yields that $H \in N^{1,1}( \widetilde{Q}, \widetilde{R} )$.
\end{proof}
\begin{rem}\label{rem:regularity}
    The Sobolev regularity $H^{-1} \in N^{1,2}( \widetilde{Q}, \widetilde{R} )$ is crucial in the following. Typically, the Sobolev regularity of the inverse of a Sobolev homeomorphism is a subtle issue in the metric surface setting.
    
    To highlight the issue, we recall \cite[Example 6.1]{Iko:Rom:20}. There an example of a metric surface $X$ was constructed for which there exists a $1$-Lipschitz homeomorphism $H \colon \mathbb{R}^{2} \rightarrow X$ with $\Mod \Gamma \leq \Mod H \Gamma$ for all path families, but $H^{-1} \not\in N^{1,2}( X, \mathbb{R}^{2} )$. In fact, $H$ is a local isometry outside a Cantor set $E \subset \mathbb{R} \times \left\{0\right\}$ of positive $\mathcal{L}^{1}$-measure and $H(E)$ has negligible $\mathcal{H}^{1}_{X}$-measure. The key point is that $X$ is not reciprocal; recall \Cref{defi:reciprocal}.
\end{rem}

    We define the following auxiliary function for later use:
    \begin{equation*}
        P( t )
        \coloneqq
        \left\{
        \begin{split}
            &t^{2},
            &&0 \leq t < 1,
            \\
            &\frac{ t^{2} }{ \mathcal{A}^{-1}( \log t^{2} ) },
            &&t \geq 1.
        \end{split}
        \right.
    \end{equation*}
    We note that for every $a \in \left[ 0, \infty \right)$,
    \begin{equation}
        \label{eq:important:integrability}
        P( a )
        \leq
        e^{ \mathcal{A}( K_{H} ) }
        +
        \frac{ a^{2} }{ K_{H} }
        \quad
        \text{for $\mathcal{H}^{2}_{ \mathbb{S}^{2}}$-a.e. in $\widetilde{Q}$}.
    \end{equation}
    This follows by first observing that $a^{2} < e^{ \mathcal{A}( K_{H} ) }$ implies $P( a ) \leq e^{ \mathcal{A}( K_{H} ) }$ and otherwise $P( a ) \leq \frac{ a^{2} }{ K_{H} }$.
    
    Also, for any measurable function $\widetilde{\rho} \colon \widetilde{Q} \rightarrow \left[ 0, \infty \right]$,
    \begin{equation}
        \label{eq:important:integrability:L1}
        \int_{ \widetilde{Q}}
            P( \widetilde{\rho} )
        \,d\mathcal{H}^{2}_{ \mathbb{S}^{2} }
        <
        \infty
        \quad\text{implies}\quad
        \int_{ \widetilde{Q} }
            \widetilde{\rho}
        \,d\mathcal{H}^{2}_{ \mathbb{S}^{2} }
        <
        \infty.
    \end{equation}
    The implication \eqref{eq:important:integrability:L1} follows since $\mathcal{A}'(t) t$ is increasing for large $t$ and converges to infinity as $t \rightarrow \infty$. Consequently, there exists $t_{1} \geq 1$ for which the derivative of $h(t) = e^{ \mathcal{A}(t) }/t^2$ is bounded from below by $h(t)/t$ for every $t \geq t_{1}$. This implies the existence of $t_{0} \geq 1$ such that $h(t) \geq 1$ for every $t \geq t_{0}$. This is equivalent to saying that $P( t ) \geq t$ for every $t \geq t_{0}$. This yields \eqref{eq:important:integrability:L1}.
    
    We set $K_{\psi}( x ) = \abs{ D\psi }_{g}^{2}/J( D( \psi ) )( x )$ and $K_{ \psi^{-1} }( x ) = \abs{ D( \psi^{-1} ) }_{g}^{2}/ J( D( \psi^{-1} ) )$. Observe that $K_{ \psi } = K_{ \psi^{-1} } \circ \psi$ $\mathcal{H}^{2}_{ \mathbb{S}^{2} }$-almost everywhere.
    
    We set $K_{H}( x ) = 1$ if $x \in \widetilde{Q} \cap \overline{Z}_{1}$ and $K_{H}( x ) = K_{ \psi }(x) $ in $x \in \widetilde{Q} \cap Z_{2}$. Then
    \begin{equation}
        \label{eq:exponentialintegrability:H}
        \int_{ \widetilde{Q} }
            e^{ \mathcal{A}( K_{H} ) }
        \,d\mathcal{H}^{2}_{ \mathbb{S}^{2} }
        <
        \infty.
    \end{equation}
    Also, $K_{ H^{-1} } \coloneqq \rho_{ H^{-1} }^{2}/J_{ H^{-1} }$ satisfies $K_{H} = K_{ H^{-1} } \circ H$ $\mathcal{H}^{2}_{ \widetilde{Z} }$-almost everywhere, since, outside a $\mathcal{H}^{2}_{ \widetilde{Z} }$-negligible set, either the number is one, or $\rho_{ H^{-1} }^2 \circ \widetilde{\iota}_{2} = \abs{ D( \psi^{-1} ) }_{g}^2$, $J_{ H^{-1} } \circ \widetilde{\iota}_{2} = J( D( \psi^{-1} ) )$, and $K_{ \psi } = K_{ \psi^{-1} } \circ \psi$.
    
    For every $z \in \widetilde{Q}$ and every pair $0 < r < r_{0}$, we denote $\Gamma( z, r, r_{0} ) \coloneqq \Gamma( \overline{B}_{\mathbb{S}^{2}}( z, r ), \widetilde{Q} \setminus B_{\mathbb{S}^{2}}( z, r_0 ); \widetilde{Q} )$.
    \begin{lemm}\label{lemm:annuluscondenser}
    For every $z \in \widetilde{Q}$ and $0 < r < r_{0}$ with $\widetilde{Q} \setminus B_{ \mathbb{S}^{2} }( z, r_{0} ) \neq \emptyset$, 
    \begin{equation}
        \label{eq:negligiblemodulus:point}
        \Mod H \Gamma( z, r, r_{0} )
        \leq
        \inf
        \left\{
            \int_{ \widetilde{Q} }
                \widetilde{\rho}^{2} K_{H}
            \,d\mathcal{H}^{2}_{ \mathbb{S}^{2} }
            \colon
            \text{ $\widetilde{\rho}$ is admissible for $\Gamma( z, r, r_{0} )$}
        \right\}.
    \end{equation}
    \end{lemm}
    \begin{proof}
    Fix an admissible function $\widetilde{\rho}$ for $\Gamma( z, r, r_{0} )$. Then for almost every $\gamma \in H \Gamma( z, r, r_{0} )$, $H^{-1} \circ \gamma$ is absolutely continuous, and
    \begin{equation*}
        1
        \leq
        \int_{ H^{-1} \circ \gamma }
            \widetilde{\rho}
        \,ds
        \leq
        \int_{ \gamma }
            ( \widetilde{\rho} \circ H^{-1} ) \rho_{H^{-1}}
        \,ds.
    \end{equation*}
    In particular, $\rho = ( \widetilde{\rho} \circ H^{-1} ) \rho_{H^{-1}}$ is weakly admissible for $H\Gamma( z, r, r_{0} )$. Consequently,
    \begin{equation*}
        \Mod H \Gamma( z, r,r_{0} )
        \leq
        \int_{ \widetilde{R} }
            \rho^{2}
        \,d\mathcal{H}^{2}_{\widetilde{Z}}.
    \end{equation*}
    The change of variables formula for $H$ and the fact that the seam $Q( S_{Z} )$ is $\mathcal{H}^{2}_{\widetilde{Z}}$-negligible establish the claim, after taking the infimum over such $\widetilde{\rho}$.
    \end{proof}

    Having observed \Cref{lemm:annuluscondenser} and \eqref{eq:exponentialintegrability:H}, the capacitary estimate \cite[Theorem 5.3]{Ko:On:06} implies that keeping $r_{0}$ fixed in \eqref{eq:negligiblemodulus:point}, we obtain $\Mod H\Gamma( z, r, r_{0} ) \rightarrow 0$ as $r \rightarrow 0^{+}$. A key point is that $\mathcal{A}$ in \eqref{eq:exponentialintegrability:H} is admissible. Since $H$ is a homeomorphism, this implies that \eqref{point:zero:modulus} holds for every $y \in \mathrm{int}( \widetilde{R} ) \subset \widetilde{Z}$. By repeating the argument with a slightly larger $\widetilde{Q}$, we conclude the following.
    \begin{lemm}\label{lemm:points:negligible}
    The identity \eqref{point:zero:modulus} holds for every $y \in \widetilde{R} \subset \widetilde{Z}$.
    \end{lemm}
    
    Fix a decomposition $\widetilde{\xi}_{1}, \widetilde{\xi}_{2}, \widetilde{\xi}_{3}, \widetilde{\xi}_{4}$ of $\partial \widetilde{Q}$ of four arcs overlapping only at their end points, labelled in cyclic order consistently with the orientation of $\mathbb{S}^{2}$. For each $i$, we denote $\xi_{i} = H( \widetilde{\xi}_{i} )$.
    
    Given the validity of \eqref{point:zero:modulus} for each $y \in \widetilde{R}$ and the universal lower bound \eqref{eq:lower:bound}, \cite[Proposition 9.1]{Raj:17} yields the existence of a homeomorphism $f = (u,v) \colon \widetilde{R} \rightarrow \left[0, 1\right] \times \left[0, M\right]$ with the following properties:
    \begin{itemize}
        \item $u \in N^{1,2}(\widetilde{R})$ with $2E(u) \eqqcolon M$ \cite[Section 4]{Raj:17};
        \item $u^{-1}( 0 ) = \xi_{1}$, $u^{-1}( 1 ) = \xi_{3}$, $v^{-1}( 0 ) = \xi_{2}$, and $v^{-1}( M ) = \xi_{4}$ \cite[Theorem 5.1 and Proposition 7.3]{Raj:17};
        \item The minimal weak upper gradient $\rho_{u}$ is weakly admissible for the path family $\Gamma( \xi_{1}, \xi_{3}; \widetilde{R} )$ and is a minimizer, i.e., $M = \Mod \Gamma( \xi_{1}, \xi_{3}; \widetilde{R} )$ \cite[Section 4-5]{Raj:17};
        \item For every Borel set $E \subset \widetilde{R}$, $\mathcal{L}^{2}( f(E) ) = \int_{ E } \rho_{u}^{2} \,d\mathcal{H}^{2}_{\widetilde{Z}}$. In particular, the Jacobian of $f$ coincides with $\rho_{u}^{2}$ \cite[Proposition 8.2]{Raj:17}.
    \end{itemize}
    The third point implies that if $u' \in N^{1,2}(\widetilde{R})$ has the same boundary values as $u$ in $\xi_{1} \cup \xi_{3}$, the Dirichlet energies satisfy $E( u ) \leq E( u' )$. Given this, we say that $u$ is an \emph{energy minimizer} for $\Gamma( \xi_{1}, \xi_{3}; \widetilde{R} )$.
    
    During the proof of \Cref{prop:parametrization}, the Beltrami differential of $H$ is defined to be zero in $\mathrm{int}( \widetilde{Q} ) \cap \overline{ Z_{1} }$, and coincide with the one of $\psi$ in $\mathrm{int}( \widetilde{Q} ) \cap Z_{2}$.
    
    \begin{prop}\label{prop:parametrization}
    The map $f = (u, v) \colon \widetilde{R} \rightarrow \left[0, 1\right] \times \left[0, M\right]$ is a $1$-quasiconformal homeomorphism.
    \end{prop}
    
    The proof of \Cref{prop:parametrization} is split into several lemmas.
    \begin{lemm}\label{lemm:conjugate}
    Let $0 < a < b < 1$ and $0 < c < d < M$ for which
    \begin{align*}
        Q^{0} &= \left\{ x \in \widetilde{R} \colon f(x) \in \left[a,b\right]\times\left[c,d\right] \right\} \subset \mathrm{int}( \widetilde{R} ) \setminus Q( S_{Z} ).
    \end{align*}
    Then $f|_{ \mathrm{int}( Q^0 ) }$ is a $1$-quasiconformal homeomorphism.
    \end{lemm}
    \begin{proof}
    For the duration of the proof, we denote
    \begin{align*}
        \xi_{1}^{0} &= f^{-1}( \left\{a\right\} \times \left[c,d\right] ),
        \quad
        \xi_{2}^{0} = f^{-1}( \left[0, 1\right] \times \left\{c\right\} ),
        \\
        \xi_{3}^{0} &= f^{-1}( \left\{b\right\} \times \left[c,d\right] ),
        \quad
        \xi_{4}^{0} = f^{-1}( \left[0, 1\right] \times \left\{d\right\} ). 
    \end{align*}
    There exists a Jordan domain $V \subset \mathrm{int}( Q ) \cap Z_{i}$, for some $i = 1,2$, such that $\widetilde{\iota}_{i}( \overline{V} ) = Q^{0}$. Equation (57) \cite[Lemma 10.2]{Raj:17} states that
    \begin{equation*}
        \Mod \Gamma( \xi_{1}^{0}, \xi_{3}^{0}; Q^{0} ) = \frac{ d - c }{ b - a }.
    \end{equation*}
    Since $\widetilde{ \iota_{i} }$ is $1$-Lipschitz and a local isometry in $\overline{V}$, we have for every quadrilateral $Q' \subset Q^{0}$,
    \begin{equation}
        \label{eq:1-reciprocality:prime}
        \Mod \Gamma( \xi_{1}', \xi_{3}'; Q' ) \Mod \Gamma( \xi_{2}', \xi_{4}'; Q' ) = 1.
    \end{equation}
    In particular, we have
    \begin{equation}
        \label{eq:1-reciprocality}
        \Mod \Gamma( \xi_{1}^{0}, \xi_{3}^{0}; Q^{0} ) \Mod \Gamma( \xi_{2}^{0}, \xi_{4}^{0}; Q^{0} ) = 1.
    \end{equation}
    We wish to apply \cite[Proposition 11.1]{Raj:17}. There Rajala assumes that \eqref{upper:bound} holds for some $\kappa \geq 1$ and concludes that $2000 \cdot \sqrt{\kappa} \rho_{u}$ is a weak upper gradient of $f$. We do not assume this. However, a quick inspection of the proof shows that given any open set $\Omega \subset \mathrm{int}( Q^{0} )$, the property \eqref{eq:1-reciprocality:prime} implies that $2000 \cdot \chi_{\mathrm{int}(Q^0)} \cdot \rho_{u}$ is a weak upper gradient of $f|_{ \mathrm{int}(Q^0) }$ in $\Omega$. By exhausting $\mathrm{int}( Q^{0} )$ by such open sets, we conclude that $f|_{ \mathrm{int}(Q^{0}) } \in N^{1,2}( \mathrm{int}(Q^0); \mathbb{R}^{2} )$.
    
    Since $u \in N^{1,2}( \widetilde{R} )$ is a continuous energy minimizer, the composition $u \circ \iota_{i}|_{ V }$ is harmonic \cite[Weyl's lemma]{Ast:Iwa:Mar:09}. The Riemann mapping theorem, the Sobolev regularity of $f|_{ \mathrm{int}(Q^{0}) }$, the boundary values of the components of $f|_{ Q^{0} }$, and \eqref{eq:1-reciprocality} imply that $f \circ \iota_{i}|_{ V }$ is a Riemann map. In particular, $f|_{ \mathrm{int}( Q^{0} ) }$ is a $1$-quasiconformal homeomorphism.
    \end{proof}
    
    \begin{lemm}\label{lemm:pullbackregularity}
    The composition $\widetilde{f} = f \circ H \colon \widetilde{Q} \rightarrow \left[0, 1\right] \times \left[0, M\right]$ is an element of $N^{1,1}( \mathrm{int}( \widetilde{Q} ), \mathbb{R}^{2} )$. Moreover, the Beltrami differential of $\widetilde{f}$ coincides with the one of $H$ and \eqref{eq:exponentialintegrability:H} holds for $K_{ \widetilde{f} }$ in place of $K_{ H }$.
    \end{lemm}
    \begin{proof}
    Given \Cref{lemm:conjugate}, the Beltrami differential of $\widetilde{f}$ and $H$ coincide $\mathcal{H}^{2}_{ \mathbb{S}^{2} }$-almost everywhere in $\mathrm{int}( \widetilde{Q} ) \setminus \mathbb{S}^{1}$, i.e., $\mathcal{H}^{2}_{ \mathbb{S}^{2} }$-almost everywhere in $\mathrm{int}( \widetilde{Q} )$. The result also implies that the pointwise distortions of $\widetilde{f}$ and $H$ coincide $\mathcal{H}^{2}_{ \mathbb{S}^{2} }$-almost everywhere in $\mathrm{int}( \widetilde{Q} )$.
    
    Next, we show that $\widetilde{u} = u \circ H \in N^{1,1}( \widetilde{Q} )$. We recall that $H \in N^{1,1}( \widetilde{Q}, \widetilde{R} )$. Moreover, if $\rho_{0} \in L^{2}( \widetilde{R} )$ is an upper gradient of $u$, the function $\rho = ( \rho_{0} \circ H ) \rho_{H}$ is a $1$-weak upper gradient of $\widetilde{u}$ with
    \begin{equation*}
        \int_{ \widetilde{Q} }
            P( \rho )
        \,d\mathcal{H}^{2}_{Z}
        \leq
        \int_{ \widetilde{Q} }
            e^{ \mathcal{A}( K_{H} ) }
        \,d\mathcal{H}^{2}_{Z}
        +
        \norm{ \rho_{0} }_{ L^{2}(Q) }^{2}
        <
        \infty,
    \end{equation*}
    where we apply \eqref{eq:important:integrability} and the distortion inequality $\rho_{H}^{2} \leq K_{H} J_{H}$. The $L^{1}( \widetilde{Q} )$-integrability of $\rho$ follows from \eqref{eq:important:integrability:L1}, so $\widetilde{u} \in N^{1,1}( \widetilde{Q} )$.
    
    Let $\widetilde{v} = v \circ H$. \Cref{lemm:conjugate} implies that $\rho = ( \rho_{0} \circ H ) \rho_{H} \in L^{1}( \widetilde{Q} )$ is a $1$-weak upper gradient of $\widetilde{v}$ in every open $U \subset \mathrm{int}( \widetilde{Q} ) \setminus \mathbb{S}^{1}$. Therefore, $\widetilde{v} \in N^{1,1}( \mathrm{int}( \widetilde{Q} ) \setminus \mathbb{S}^{1} )$. Given the continuity of $\widetilde{v}$, we actually have $\widetilde{v} \in N^{1,1}( \mathrm{int}( \widetilde{Q} ) )$. This is seen by verifying the ACL (absolute continuity on lines) property for $\widetilde{v}|_{ \mathrm{int}( \widetilde{Q} ) }$ on charts covering $\mathbb{S}^{1} \cap \mathrm{int}( \widetilde{Q} )$. The ACL property on charts follows from a minor modification of the proof in \cite[Theorem 35.1]{Vai:71} showing that closed sets with $\sigma$-finite Hausdorff $1$-measure are quasiconformally removable. This implies that $\rho$ is a $1$-weak upper gradient of $\widetilde{v}$ on $\mathrm{int}( \widetilde{Q} )$. The claim follows from this.
    \end{proof}
    
    \begin{lemm}\label{lemm:components}
    Let $u'$ denote the energy minimizer for $\Gamma( \xi_{2}, \xi_{4}; Q )$. Then $v = M u'$.
    \end{lemm}
    \begin{proof}
    Similarly to $f$ and $\widetilde{f}$, let $f' = (u', v')$ and $\widetilde{f}'$ denote the homeomorphisms obtained from the energy minimizer $u'$ for $\Gamma( \xi_{2}, \xi_{4}; \widetilde{R} )$. Let $R'$ denote the image of $f'$ and $R$ the image of $f$.
    
    \Cref{lemm:pullbackregularity} shows that the Beltrami differentials of $\widetilde{f}$ and $\widetilde{f}'$ coincide with one another $\mathcal{H}^{2}_{ \mathbb{S}^{2} }$-almost everywhere and their distortion satisfies \eqref{eq:exponentialintegrability:H} for an admissible $\mathcal{A}$. Then the Stoilow factorization theorem \cite[Theorems 20.5.1, 20.5.2]{Ast:Iwa:Mar:09} implies that $\varphi = \widetilde{f}' \circ \widetilde{f}^{-1}$ is conformal; note also that $\varphi = f' \circ f^{-1}$.
    
    Since $\varphi$ is conformal, the energy minimizer $\pi_{1}$ for $\Gamma( f'(\xi_{2}), f'(\xi_{4}); R' )$ is such that $\pi_{1} \circ \varphi$ is the energy minimizer for $\Gamma( f(\xi_{2}), f(\xi_{4}); R )$. On the other hand, here $\pi_{1}$ is the projection to the $x$-axis and $\pi_{1} \circ \varphi$ is $M^{-1}$ times the projection to the $y$-axis. Since $\varphi = f' \circ f^{-1}$, the equality $u' = \pi_{1} \circ \varphi \circ f = M^{-1} v$ follows.
    \end{proof}
    
    \begin{proof}[Proof of \Cref{prop:parametrization}]
    \Cref{lemm:components} implies that $f = (u,v) \in N^{1,2}( \widetilde{R}, \mathbb{R}^{2} )$. Furthermore, \Cref{lemm:conjugate} implies $\rho_{f}^{2} = J_{f} \in L^{1}( \widetilde{R} )$. Hence $\Mod \Gamma \leq \Mod f \Gamma$ for every path family in $\widetilde{R}$. This improves to $K$-quasiconformality for some $K \geq 1$ due to \Cref{prop:outer:to:maximal}. As $f( Q(S_{Z}) \cap \widetilde{R} )$ is negligible due to the change of variables formula for $f$, and as $f^{-1}$ is $1$-quasiconformal outside $f( S_{Z} \cap \widetilde{R} )$, we immediately obtain $\Mod \Gamma \leq \Mod f^{-1} \Gamma$ for every path family in $f( \widetilde{R} )$. Thus $f$ is $1$-quasiconformal.
    \end{proof}
    
    \begin{proof}[Proof of \Cref{thm:local:recip}]
    This is proved by \Cref{prop:parametrization}.
    \end{proof}
    
    \begin{rem}\label{rec:localization:scheme}
    Notice that if \Cref{lemm:points:negligible} holds for a given homeomorphism $g \colon \mathbb{S}^{1} \rightarrow \mathbb{S}^{1}$ having an admissible extension, even without assuming the absolutely continuity of $g^{-1}$, the rest of the proof of \Cref{prop:parametrization} (and \Cref{thm:local:recip}) go through the same way.
    \end{rem}

\begin{proof}[Proof of \Cref{thm:QS:abs}]
    Given a quasisymmetry $g \colon \mathbb{S}^{1} \rightarrow \mathbb{S}^{1}$, its \emph{Beurling--Ahlfors} extension $\psi \colon \overline{ Z }_{2} \rightarrow \overline{ Z }_{2}$ is a quasisymmetry and $\psi|_{ Z_{2} }$ is $K$-quasiconformal for some $K \geq 1$ \cite{Ah:Beu:56}. Thus, if $g^{-1}$ is absolutely continuous, $g$ satisfies the assumptions of \Cref{thm:biconformal:abs}. Alternatively, if $H$ is as in \Cref{defi:standingassumptions}, \Cref{lemm:annuluscondenser} implies that $H^{-1}$ has outer dilatation $K_{O}( H^{-1} ) \leq K$. \Cref{prop:outer:to:maximal} implies that $H$ is quasiconformal; this self-improves to $K$-quasiconformality. Clearly $H$ extends to a $K$-quasiconformal homeomorphism $H \colon \mathbb{S}^{2} \rightarrow \widetilde{Z}$.
\end{proof}

\section{Concluding remarks}\label{sec:obstructions}
\subsection{A point of positive capacity}\label{sec:nonexample1}
    For a general orientation-preserving homeomorphism $g \colon \mathbb{S}^{1} \rightarrow \mathbb{S}^{1}$, the $\widetilde{Z}$ can have points of positive capacity (in the sense that \eqref{point:zero:modulus} can fail) even if $g$ is locally bi-Lipschitz in the complement of a single point. For example, having fixed arbitrary $1 < \alpha < \beta$, we consider the homeomorphism $h \colon \mathbb{R} \rightarrow \mathbb{R}$ defined by
    \begin{equation}
        \label{eq:badexample}
        h(x)
        =
        \left\{
        \begin{split}
            &x^{\alpha},
            \quad
            &&\text{$x \geq 0$},
            \\
            &-(-x)^{\beta}, 
            \quad
            &&\text{$x < 0$}.
        \end{split}
        \right.
    \end{equation}
    We construct a homeomorphism $g \colon \mathbb{S}^{1} \rightarrow \mathbb{S}^{1}$ by restricting $h$ to the interval $\left[-1, 1\right]$, extending the restriction to $\mathbb{R}$ periodically, and by considering the covering map $\theta(t) = ( \cos( \pi t ), \sin( \pi t), 0 )$, and a homeomorphism $g \colon \mathbb{S}^{1} \rightarrow \mathbb{S}^{1}$ satisfying $g \circ \theta = \theta \circ h^{-1}$. Then $g^{-1}$ is an $L$-Lipschitz homeomorphism for some $L \geq 1$, and one can check directly from the definition of $d_{Z}$ that the inclusion map $\widetilde{ \iota }_{1} \colon \overline{Z}_{1} \rightarrow \widetilde{Z}$ is $L$-bi-Lipschitz onto its image.

    Let $x_{0} \in \widetilde{Z}$ denote the point corresponding to $( 1, 0, 0 )$. By using the techniques from \Cref{sec:Ahlfors}, we can show that $\widetilde{Z} \setminus \left\{ x_{0} \right\}$ can be covered by bi-Lipschitz images of planar domains. Then \cite[Theorem 1.3]{Iko:19} implies that $\widetilde{Z} \setminus \left\{ x_{0} \right\}$ is $1$-quasiconformally equivalent to a Riemannian surface (that is homeomorphic to a planar domain). Such a Riemannian surface can be conformally embedded into $\mathbb{S}^{2}$ \cite[Section III.4]{Ah:Sa:60}. Hence there exists a $1$-quasiconformal embedding $\psi \colon \widetilde{Z} \setminus \left\{ x_0 \right\} \rightarrow \mathbb{S}^{2}$.
    
    We claim that the complement of the image of $\psi$ is a non-trivial continuum (which is equivalent to the failure of \eqref{point:zero:modulus} at $x_{0}$). Indeed, otherwise $\psi$ would extend to a $1$-quasiconformal homeomorphism and $g$ would be a welding homeomorphism, as a consequence of \Cref{thm:welding:positive}. This would contradict both \cite[Example 1]{Oik:61} and \cite[Theorem 3]{Vai:89}, where both of these result show that $g$ is not a welding homeomorphism.
    
    In contrast, if we set $\alpha = \beta \geq 1$ in \eqref{eq:badexample}, the homeomorphism $g$ is a quasisymmetry, so $\widetilde{Z}$ is quasiconformally equivalent to $\mathbb{S}^{2}$, as a consequence of \Cref{thm:QS:abs}.

\subsection{Points of positive capacity}\label{sec:nonexample2}
    We construct another example for which points of positive capacity occur. To this end, consider a Cantor set $E \subset \left[0, 1\right]$ and
    \begin{equation}
        \label{eq:Cantor}
        h(x)
        =
        \left\{
        \begin{split}
            &( \mathcal{L}^{1}( \left[0,1\right] \setminus E ) )^{-1}
            \int_{ 0 }^{ x } \chi_{ \mathbb{R} \setminus E }(y) \,d\mathcal{L}^{1}(y),
            \quad
            &&\text{$0 \leq x \leq 1$},
            \\
            &x, 
            \quad
            &&\text{otherwise}.
        \end{split}
        \right.
    \end{equation}
    Then $h \colon \mathbb{R} \rightarrow \mathbb{R}$ is a Lipschitz homeomorphism coinciding with the identity map outside $\left(0,1\right)$.
    
    Next, consider the Möbius transformation $\theta_{1}( z ) = (z-i)/(z+i)$ from the upper half-space $\overline{ \mathbb{H} }$ onto the Euclidean unit disk $\overline{ \mathbb{D} }$. Let $\theta_{2}(x,y) = ( 2x/(1+x^2+y^2), 2y/(1+x^2+y^2), ( 1 - x^2 - y^2 )/( 1 + x^2 + y^2 ) )$. Then $\theta \coloneqq \theta_{2} \circ \theta_{1} \colon \overline{ \mathbb{H} } \rightarrow \overline{Z}_{2}$ defines a $1$-quasiconformal homeomorphism, given that $\theta_{2}^{-1}$ is a(n orientation-reversing) stereographic projection.
    
    There exists a unique homeomorphism $g \colon \mathbb{S}^{1} \rightarrow \mathbb{S}^{1}$ satisfying $g \circ \theta = \theta \circ h^{-1}$. We see from \eqref{eq:Cantor} that $g^{-1}$ is $L$-Lipschitz and $\widetilde{\iota}_{1}$ is $L$-bi-Lipschitz with a constant $L$ depending only on $\mathcal{L}^{1}( E )$. In particular, $\widetilde{Z} = ( Z, d_{Z} )$.
    
    We denote $E' = \widetilde{\iota}_{2}( \theta(E) ) \subset \widetilde{Z}$, and apply \cite[Theorem 1.3]{Iko:19} as in \Cref{sec:nonexample1}, and find a $1$-quasiconformal embedding $\psi \colon \widetilde{Z} \setminus E' \rightarrow \mathbb{S}^{2}$.
    
    Consider on $\mathbb{R}^{2}$ the distance $d_{E}$ obtained as follows: For each absolutely continuous $\gamma \colon \left[0, 1\right] \rightarrow \mathbb{R}^{2}$, denote $\ell_{E}( \gamma ) \coloneqq \int_{ \gamma } \chi_{ \mathbb{R}^{2} \setminus E } \,ds$. We set $d_{E}( x, y ) = \inf \ell_{E}( \gamma )$, the infimum taken over absolutely continuous paths joining $x$ to $y$.
    
    We denote $X = ( \mathbb{R}^{2}, d_{E} )$. The change of distance map $H \colon \mathbb{R}^{2} \rightarrow X$ is a $1$-Lipschitz homeomorphism that is a local isometry on $\mathbb{R}^{2} \setminus E$. Moreover, if $\theta \colon \left[0, 1\right] \rightarrow \mathbb{R}^{2}$ is absolutely continuous, the metric speeds satisfy
    \begin{equation}
        \label{eq:metricspeed:image}
        v_{ H \circ \theta }
        =
        \left( \chi_{ \mathbb{R}^{2} \setminus E } \circ \theta \right) \cdot v_{\theta}
        \quad\text{$\mathcal{L}^{1}$-almost everywhere}.
    \end{equation}
    The composition $G = \widetilde{\iota}_{2} \circ \theta \circ \left( H|_{ \left[-1, 2\right] \times \left[0,1\right] } \right)^{-1}$ is a $1$-quasiconformal homeomorphism. This follows from \Cref{lemm:techincal:lemma}, the equalities $\mathcal{H}^{1}_{ \widetilde{Z} }( E' ) = 0 = \mathcal{H}^{1}_{ X }( H(E) )$, together with \Cref{lemm:hausdorff} and \eqref{eq:metricspeed:image}.
    
    We consider a Cantor set $E$ obtained from \cite[Example 6.1]{Iko:Rom:20}. The key property of $E$ is the following: there exists a path family $\Gamma$ on $\left[0, 1\right]^{2}$, each path joining $(0,0)$ to $(1,0)$, such that $\Mod H \Gamma \geq ( 4 \pi )^{-1}$ and $\Mod \Gamma = 0$. Given that $G$ is $1$-quasiconformal, the points $\widetilde{\iota}_{2}( \theta(x) )$, where $x = (0,0), (1,0)$, fail \eqref{point:zero:modulus}. Consequently, $\widetilde{Z}$ is not quasiconformally equivalent to $\mathbb{S}^{2}$, and the embedding $\psi$ does not have a quasiconformal extension $\Psi \colon \widetilde{Z} \rightarrow \mathbb{S}^{2}$.
    \begin{ques}\label{ques:extension}
    Are there Cantor sets $E$ with $\mathcal{L}^{1}( E ) > 0$ such that a quasiconformal embedding $\psi \colon \widetilde{Z} \setminus E' \rightarrow \mathbb{S}^{2}$ extends to a quasiconformal homeomorphism $\Psi \colon \widetilde{Z} \rightarrow \mathbb{S}^{2}$?
    \end{ques}
    Given a compact set $F \subset Y$ with $Y = \mathbb{R}^{2}$ or $Y = \mathbb{S}^{2}$, we say that $F$ has \emph{zero absolute area} if every $1$-quasiconformal embedding $f \colon Y \setminus F \rightarrow \mathbb{S}^{2}$ satisfies $\mathcal{H}^{2}_{ \mathbb{S}^{2} }( \mathbb{S}^{2} \setminus f( Y \setminus F ) ) = 0$.
    
    We expect that the quasiconformal extension $\Psi$ exists if and only if the set $F = \mathbb{S}^{2} \setminus \psi( \widetilde{Z} \setminus E' )$ has zero absolute area; the "only if"-direction follows by applying the techniques used in \Cref{sec:harm:weld}, by noting that the composition $( f \circ \psi )^{-1}$ has a continuous, monotone, and surjective extension $\widetilde{\pi}$ with $\Mod \Gamma \leq \Mod \widetilde{\pi} \Gamma$ for all path families. We expect that the "if"-direction follows from \cite[Theorems 1.3 and 1.4, together with Lemma 5.1]{Iko:Rom:20}.

    If $E$ in \Cref{ques:extension:reformulate} has zero absolute area, \cite[Theorem 1.3]{Iko:Rom:20} implies that the change of distance map $H$ is a $1$-quasiconformal homeomorphism. Given that the $G$ above is $1$-quasiconformal, one readily verifies that $\widetilde{\iota}_{2}$ is a $1$-quasiconformal homeomorphism onto its image. We ask the following.
    \begin{ques}\label{ques:extension:reformulate}
    Let $E$, $g$, and $\psi$ be as in \Cref{ques:extension}. If $\widetilde{\iota}_{2} \colon \overline{Z}_{2} \rightarrow \widetilde{Z}$ is a $1$-quasiconformal parametrization of its image, does $\psi \colon \widetilde{Z} \setminus E' \rightarrow \mathbb{S}^{2}$ extend to a quasiconformal homeomorphism $\Psi \colon \widetilde{Z} \rightarrow \mathbb{S}^{2}$?
    
    In particular, if $E$ has zero absolute area, does $F = \mathbb{S}^{2} \setminus \psi( \widetilde{Z} \setminus E' )$ have zero absolute area?
    \end{ques}
    It follows from \cite[Theorem 1.1 and Proposition 1.2]{Iko:20} that the inclusion map $\widetilde{\iota}_{2}$ is a $1$-quasiconformal homeomorphism if and only if there exists a quasiconformal homeomorphism $h \colon \widetilde{\iota}_{2}( \overline{Z}_{2} ) \rightarrow \overline{ \mathbb{D} }$, where $\overline{ \mathbb{D} }$ is the closed Euclidean unit disk.

\subsection{Welding homeomorphisms}
    We consider a welding homeomorphism $g \colon \mathbb{S}^{1} \rightarrow \mathbb{S}^{1}$ with welding curve $\mathcal{C} \subset \mathbb{S}^{2}$. Consider the monotone mapping $\widetilde{\pi} \colon \mathbb{S}^{2} \rightarrow \widetilde{Z}$ obtained from \eqref{eq:quotientmap}.
    \begin{ques}\label{ques:nocollapsing}
    If $\widetilde{\pi}$ is a homeomorphism, is it a $1$-quasiconformal homeomorphism?
    \end{ques}
    We showed in \Cref{thm:failure:terrible} that if $\widetilde{\pi}$ is not a homeomorphism, then $\widetilde{Z}$ is not quasiconformally equivalent to $\mathbb{S}^{2}$; the collapsing creates points of positive capacity --- by which we mean that \eqref{point:zero:modulus} fails --- in $\widetilde{Z}$. \Cref{ques:nocollapsing} asks if the collapsing is the only obstruction for quasiconformal uniformization. \Cref{lemm:regularity} reduces the question to understanding when $\widetilde{\pi}^{-1} \in N^{1,2}( \widetilde{Z}, \mathbb{S}^{2} )$.

\subsection{Quasisymmetries}    
    Observe that the assumptions of \Cref{thm:QS:abs} are satisfied by every quasisymmetry $g \colon \mathbb{S}^{1} \rightarrow \mathbb{S}^{1}$ that is \emph{strongly quasisymmetric} \cite{Sem:86} \cite{Bis:88} \cite{Ast:Zin:91} \cite{Bis:Jon:94}: for every $\epsilon > 0$ there exists $\delta > 0$ such that for every subarc $I \subset \mathbb{S}^{1}$ and Borel set $E \subset I$,
\begin{equation*}
    \mathcal{H}^{1}_{ \mathbb{S}^{1} }( E )
    \leq
    \delta \mathcal{H}^{1}_{ \mathbb{S}^{1} }( I )
    \quad\text{implies}\quad
    \mathcal{H}^{1}_{ \mathbb{S}^{1} }( g(E) )
    \leq
    \epsilon \mathcal{H}^{1}_{ \mathbb{S}^{1} }( g(I) ).
\end{equation*}
    The welding curves corresponding to strongly quasisymmetric homeomorphisms are special cases of the \emph{asymptotically conformal} quasicircles; see \cite{Pom:78}. One might ask whether or not $\widetilde{Z}$ is quasiconformally equivalent to $\mathbb{S}^{2}$ whenever $g$ is a welding homeomorphism corresponding to such a curve. Corollary 4 of \cite{Pom:78} provides us with an example of asymptotically conformal quasicircle $\mathcal{C}$ which has an uncountable number of tangent points, with the tangent points dense in $\mathcal{C}$, but they also have zero 1-dimensional Hausdorff measure. 
    \begin{lemm}\label{lemm:asymptotic:conformal}
    There exists a quasisymmetric $g \colon \mathbb{S}^{1} \rightarrow \mathbb{S}^{1}$ with asymptotically conformal welding curve $\mathcal{C}$ such that $\widetilde{Z}$ is not homeomorphic to $\mathbb{S}^{2}$.
    \end{lemm}
    \Cref{lemm:asymptotic:conformal} follows from \Cref{thm:failure:terrible}, \Cref{lemm:mutualABS:tan}, and the cited example.
    \begin{ques}\label{ques:QS}
    Is the answer to \Cref{ques:nocollapsing} yes if we also assume that $g \colon \mathbb{S}^{1} \rightarrow \mathbb{S}^{1}$ is a quasisymmetry?
    \end{ques}
    To answer \Cref{ques:QS} negatively, one needs to construct a quasisymmetry $\psi \colon \overline{Z}_{2} \rightarrow \overline{Z}_{2}$, with $g = \psi|_{ \mathbb{S}^{1} }$, for which the measures $g^{*}\mathcal{H}^{1}_{ \mathbb{S}^{1} }$ and $\mathcal{H}^{1}_{ \mathbb{S}^{1} }$ are not mutually singular in any subarc $I \subset \mathbb{S}^{1}$, yet the corresponding $\widetilde{Z}$ is not quasiconformally equivalent to $\widetilde{Z}$. Equivalently, one only needs to show that the homeomorphism $H \colon \mathbb{S}^{2} \rightarrow \widetilde{Z}$, coinciding with $\widetilde{\iota}_{1}$ in $Z_{1}$ and with $\widetilde{\iota}_{2} \circ \psi$ in $Z_{2}$, is not quasiconformal. By arguing as in the proof of \Cref{lemm:regularity}, one sees that $H$ is quasiconformal if and only if $H^{-1} \in N^{1,2}( \widetilde{Z}, \mathbb{S}^{2} )$.
    
\bibliographystyle{alpha}
\bibliography{bibliography} 

\end{document}